\documentclass[11pt]{amsart}

\usepackage{amsmath}
\usepackage{amssymb}
\usepackage{amscd}
\usepackage{stmaryrd}
\usepackage{mathrsfs}

\usepackage{xy}
\xyoption{all}

\newcommand{\nc}{\newcommand}

\nc{\exto}[1]{\stackrel{#1}{\longrightarrow}}
\nc{\dlim}{{\mathop{\lim\limits_{\longrightarrow}\,}}}
\nc{\ilim}{{\mathop{\lim\limits_{\longleftarrow}\,}}}
\nc{\hocolim}{{\mathop{\sf hocolim}\,}}
\nc{\holim}{{\mathop{\sf holim}}}
\nc{\lan}{\big\langle}
\nc{\ran}{\big\rangle}

\nc{\kk}{{\mathsf{k}}}

\nc{\C}{{\mathbb{C}}}
\nc{\HH}{{\mathbf{H}}}
\nc{\DD}{{\mathbb{D}}}
\nc{\LL}{{\mathbb{L}}}
\nc{\PP}{{\mathbb{P}}}
\nc{\QQ}{{\mathbb{Q}}}
\nc{\RR}{{\mathbb{R}}}
\nc{\ZZ}{{\mathbb{Z}}}

\nc{\CA}{{\mathcal{A}}}
\nc{\CB}{{\mathcal{B}}}
\nc{\CC}{{\mathcal{C}}}
\nc{\D}{{\mathcal{D}}}
\nc{\SCA}{{\mathscr{A}}}
\nc{\SCB}{{\mathscr{B}}}
\nc{\SCC}{{\mathscr{C}}}
\nc{\SCD}{{\mathscr{D}}}
\nc{\SCE}{{\mathscr{E}}}
\nc{\SCH}{{\mathscr{H}}}
\nc{\CE}{{\mathcal{E}}}
\nc{\CF}{{\mathcal{F}}}
\nc{\CG}{{\mathcal{G}}}
\nc{\CH}{{\mathcal{H}}}
\nc{\CL}{{\mathcal{L}}}
\nc{\CM}{{\mathcal{M}}}
\nc{\CN}{{\mathcal{N}}}
\nc{\CO}{{\mathcal{O}}}
\nc{\CQ}{{\mathcal{Q}}}
\nc{\CR}{{\mathcal{R}}}
\nc{\CS}{{\mathcal{S}}}
\nc{\CT}{{\mathcal{T}}}
\nc{\CU}{{\mathcal{U}}}
\nc{\CV}{{\mathcal{V}}}
\nc{\CW}{{\mathcal{W}}}
\nc{\CX}{{\mathcal{X}}}
\nc{\CY}{{\mathcal{Y}}}
\nc{\CMo}{{\mathcal{M}^\circ}}
\nc{\Co}{{{C}^\circ}}

\nc{\BY}{{\overline{Y}}}
\nc{\BYD}{{\overline{Y}{}^{|D|}}}
\nc{\OZ}{{\overline{Z}}}
\nc{\bg}{{\bar{g}}}

\nc{\ba}{{\mathbf{a}}}
\nc{\bb}{{\mathbf{b}}}
\nc{\be}{{\mathbf{e}}}
\nc{\bq}{{\mathbf{q}}}
\nc{\bx}{{\mathbf{x}}}
\nc{\by}{{\mathbf{y}}}
\nc{\bz}{{\mathbf{z}}}
\nc{\BB}{{\mathbf{B}}}
\nc{\BC}{{\mathbf{C}}}
\nc{\BE}{{\mathbf{E}}}
\nc{\BD}{{\mathbf{D}}}
\nc{\BG}{{\mathbf{G}}}
\nc{\BM}{{\mathbf{M}}}
\nc{\BP}{{\mathbf{P}}}
\nc{\BU}{{\mathbf{U}}}
\nc{\BZ}{{\mathbf{Z}}}
\nc{\BPr}{{\mathsf{P}}}
\nc{\BL}{{\mathbf{L}}}
\nc{\BR}{{\mathbf{R}}}
\nc{\BRO}[1]{{{\mathbf{R}}^{\circ}_{#1}}}
\nc{\BRD}[1]{{{\mathbf{R}}^{|D|}_{#1}}}
\nc{\BRP}[1]{{{\mathbf{R}}^{1}_{#1}}}
\nc{\BRTP}[1]{{{\mathbf{\tilde{R}}}{}^{1}_{#1}}}
\nc{\BS}{{\mathbf{S}}}
\nc{\BMS}{{{\mathbf{M}}^{{s}}}}
\nc{\BMSS}{{{\mathbf{M}}^{{ss}}}}
\nc{\BMZ}{{\mathbf{M}^{\circ}}}
\nc{\BCL}{{\mathbf{L}}}

\nc{\PCC}{{{}^\perp\CC}}

\nc{\Cl}{{\mathsf{Cliff}}}
\nc{\Clev}{{\mathop{\mathsf{Cliff}}^{\circ}}}

\nc{\FA}{{\mathfrak{A}}}
\nc{\FB}{{\mathfrak{B}}}
\nc{\FU}{{\mathfrak{U}}}

\nc{\fa}{{\mathfrak{a}}}
\nc{\fb}{{\mathfrak{b}}}
\nc{\fg}{{\mathfrak{g}}}
\nc{\fn}{{\mathfrak{n}}}
\nc{\fp}{{\mathfrak{p}}}
\nc{\FD}{{\mathfrak{D}}}
\nc{\FE}{{\mathfrak{E}}}
\nc{\FL}{{\mathfrak{L}}}
\nc{\FM}{{\mathfrak{M}}}
\nc{\FS}{{\mathsf{S}}}

\nc{\se}{{\mathsf{e}}}
\nc{\sfc}{{\mathsf{c}}}
\nc{\sfch}{{\mathsf{ch}}}
\nc{\sfh}{{\mathsf{h}}}

\nc{\SK}{{\mathsf{K}}}
\nc{\SM}{{\mathsf{M}}}
\nc{\SO}{{\mathsf{O}}}
\nc{\SQ}{{\mathsf{Q}}}
\nc{\SPV}{{\mathsf{S}^+\mathsf{V}}}
\nc{\SMV}{{\mathsf{S}^-\mathsf{V}}}
\nc{\SPMV}{{\mathsf{S}^\pm\mathsf{V}}}
\nc{\SX}{{S_X}}
\nc{\SY}{{S_Y}}
\nc{\phipsi}{{q}}
\nc{\eps}{\varepsilon}

\nc{\pim}{{\pi_-}}
\nc{\pip}{{\pi_+}}

\nc{\TE}{{\tilde{\CE}}}
\nc{\TQ}{{\tilde{Q}}}
\nc{\TCF}{{\tilde{\CF}}}
\nc{\TCG}{{\tilde{\CG}}}
\nc{\TCL}{{\tilde{\CL}}}
\nc{\TF}{{\tilde{F}}}
\nc{\TW}{{\tilde{W}}}
\nc{\TCC}{{\tilde{\CC}}}
\nc{\TCX}{{\tilde{\CX}}}
\nc{\TCY}{{\tilde{\CY}}}
\nc{\TPhi}{{\tilde{\Phi}}}
\nc{\OPhi}{{\bar{\Phi}}}
\nc{\txi}{{\tilde{\xi}}}
\nc{\tp}{{\tilde{p}}}
\nc{\tq}{{\tilde{q}}}
\nc{\tzeta}{{\tilde{\zeta}}}
\nc{\tpi}{{\tilde{\pi}}}

\nc{\halpha}{{\hat{\alpha}}}
\nc{\HCA}{{\hat{\CA}}}
\nc{\HCB}{{\hat{\CB}}}
\nc{\HCC}{{\hat{\CC}}}
\nc{\HE}{{\widehat{\CE}}}
\nc{\HX}{{\hat{X}}}
\nc{\hxi}{{\hat{\xi}}}

\nc{\UH}{{\mathcal{H}}}

\nc{\TM}{{\widetilde{M}}}
\nc{\TCM}{{\widetilde{\CM}}}
\nc{\TU}{{\widetilde{U}}}
\nc{\TX}{{\widetilde{X}}}
\nc{\TY}{{\widetilde{Y}}}
\nc{\TYO}{{{\widetilde{Y}}^\circ}}
\nc{\barf}{{\bar{f}}}
\nc{\te}{{\tilde{e}}{}}
\nc{\tf}{{\tilde{f}}}
\nc{\tg}{{\tilde{g}}}
\nc{\ti}{{\tilde{\imath}}}
\nc{\tj}{{\tilde{\jmath}}}
\nc{\ty}{{\tilde{y}}}
\nc{\tphi}{{\tilde{\phi}}}

\nc{\urho}{{\underline{\rho}}}

\nc{\LRA}{\Leftrightarrow}
\nc{\RA}{\Rightarrow}
\nc{\lotimes}{\mathbin{\mathop{\otimes}\limits^{\mathbb{L}}}}
\nc{\CEnd}{\mathop{\mathcal{E}\mathit{nd}}\nolimits}
\nc{\CExt}{\mathop{\mathcal{E}\mathit{xt}}\nolimits}
\nc{\CHom}{\mathop{\mathcal{H}\mathit{om}}\nolimits}
\nc{\RH}{\mathop{{\mathsf{R}}\Gamma}\nolimits}
\nc{\RGamma}{\mathop{{\mathsf{R}}\Gamma}\nolimits}
\nc{\RHom}{\mathop{\mathsf{RHom}}\nolimits}
\nc{\RCHom}{\mathop{\mathsf{R}\mathcal{H}\mathit{om}}\nolimits}
\nc{\RG}{\mathop{\mathsf{R\Gamma}}\nolimits}
\nc{\Hom}{\mathop{\mathsf{Hom}}\nolimits}
\nc{\Ext}{\mathop{\mathsf{Ext}}\nolimits}
\nc{\End}{\mathop{\mathsf{End}}\nolimits}
\nc{\Tor}{\mathop{\mathsf{Tor}}\nolimits}
\nc{\Tordim}{\mathop{\mathsf{Tor}\text{\rm-}\mathsf{dim}}\nolimits}
\nc{\Hilb}{\mathop{\mathsf{Hilb}}\nolimits}
\nc{\Spec}{\mathop{\mathsf{Spec}}\nolimits}
\nc{\Pic}{\mathop{\mathsf{Pic}}\nolimits}

\nc{\Tr}{\mathop{\mathsf{Tr}}\nolimits}
\nc{\tr}{\mathop{\mathsf{tr}}\nolimits}
\nc{\sd}{\mathop{\mathsf{sd}}\nolimits}
\nc{\LInd}{\mathop{\mathsf{LInd}}\nolimits}
\nc{\Res}{\mathop{\mathsf{Res}}\nolimits}
\nc{\Cone}{\mathop{\mathsf{Cone}}\nolimits}
\nc{\Fiber}{\mathop{\mathsf{Fiber}}\nolimits}
\nc{\Ker}{\mathop{\mathsf{Ker}}\nolimits}
\nc{\Coker}{\mathop{\mathsf{Coker}}\nolimits}
\nc{\codim}{\mathop{\mathsf{codim}}\nolimits}
\nc{\sing}{{\mathsf{sing}}}
\nc{\supp}{\mathop{\mathsf{supp}}}
\nc{\vol}{\mathop{\mathsf{vol}}\nolimits}
\nc{\ch}{\mathop{\mathsf{ch}}\nolimits}
\nc{\perf}{{\mathsf{perf}}}
\nc{\rank}{\mathop{\mathsf{rank}}}
\nc{\Pf}{{\mathsf{Pf}}}
\nc{\Gr}{{\mathsf{Gr}}}
\nc{\OGr}{{\mathsf{OGr}}}
\nc{\Flag}{{\mathsf{Fl}}}
\nc{\Kosz}{{\mathsf{Kosz}}}
\nc{\LGr}{{\mathsf{LGr}}}
\nc{\GTGr}{{\mathsf{G_2Gr}}}
\nc{\GTF}{{\mathsf{G_2F}}}
\nc{\OF}{{\mathsf{OF}}}
\nc{\Fl}{{\mathsf{Fl}}}
\nc{\Bl}{{\mathsf{Bl}}}
\nc{\GL}{{\mathsf{GL}}}
\nc{\PGL}{{\mathsf{PGL}}}
\nc{\SL}{{\mathsf{SL}}}
\nc{\SP}{{\mathsf{Sp}}}
\nc{\Spin}{{\mathsf{Spin}}}
\nc{\Tot}{{\mathsf{Tot}}}
\nc{\ev}{{\mathsf{ev}}}
\nc{\od}{{\mathsf{odd}}}
\nc{\coev}{{\mathsf{coev}}}
\nc{\id}{{\mathsf{id}}}
\nc{\opp}{{\mathsf{opp}}}
\nc{\PS}{{{\PP^3}}}
\nc{\Qu}{{{Q^3}}}
\nc{\tdim}{\mathop{\Tor\dim}}
\nc{\ecart}{{\fbox{$\scriptstyle\mathsf{EC}$}}}
\nc{\ad}{{\mathop{\mathsf ad}}}
\nc{\sg}{{\mathop{\mathsf sg}}}
\nc{\hf}{{\mathop{\mathsf hf}}}
\nc{\gr}{{\mathop{\mathsf gr}}}
\nc{\qgr}{{\mathop{\mathsf qgr}}}
\nc{\Perf}{{\mathop{{\mathsf{Perf}}}}}
\nc{\Coh}{{\mathop{{\mathsf{Coh}}}}}
\nc{\coh}{{\mathop{{\mathsf{coh}}}}}
\nc{\dgm}{{\mathop{{\mathsf{dgm}\text{-}}}}}
\nc{\acy}{{\mathop{{\mathsf{acycl}\text{-}}}}}
\nc{\hproj}{{\mathop{{\mathsf{hproj}\text{-}}}}}
\nc{\com}{{\mathop{{\mathsf{com}}}}}
\nc{\comp}{{\mathop{{\mathsf{comp}}}}}
\nc{\Ab}{{\mathop{\mathcal{A}\mathit{b}}}}
\nc{\Open}{{\mathop{\mathsf{Open}}}}
\nc{\Ccoh}{{\mathop{\mathsf Ccoh}}}
\nc{\Qcoh}{{\mathop{\mathsf{Qcoh}}}}
\nc{\arcs}{{\mathop{\mathsf{arcs}}}}
\nc{\Shuf}{{\mathop{\mathsf{Shuf}}}}
\nc{\circles}{{\mathop{\mathsf{circles}}}}
\nc{\normcircles}{\underline{\mathop{\mathsf{circles}}}}
\nc{\At}{{\mathop{\mathsf{At}}\nolimits}}
\nc{\tra}{{\mathsf{T}}}
\nc{\fsl}{{\mathfrak{sl}}}
\nc{\fso}{{\mathfrak{so}}}
\nc{\fgl}{{\mathfrak{gl}}}

\nc{\AAV}{{\mathcal{AAV}}}

\nc{\Rep}{{\mathsf{Rep}}}

\nc{\Cubics}{{{\mathcal{S}}_3}}
\nc{\VFT}{{{\mathcal{S}}_{14}}}
\nc{\VFTE}{{{\mathcal{N}}_{\mathrm{reg,sm}}}}
\nc{\MX}{{\CM_X}}
\nc{\MY}{{\CM_Y}}
\nc{\MYE}{{\CM_{Y,\CE}}}
\nc{\Yd}{{Y_d}}
\nc{\Yfive}{{Y_5}}
\nc{\Xg}{{X_{2g-2}}}
\nc{\Xtt}{{X_{22}}}
\nc{\Xst}{{X_{16}}}
\nc{\Xtw}{{X_{12}}}
\nc{\Xe}{{X_{8}}}
\nc{\Xf}{{X_{4}}}

\nc{\git}{{/\!\!/\!{}_\chi}}

\nc{\HOH}{{\mathsf H\mathsf H}}
\nc{\NHH}{{\mathsf N\HOH}}
\nc{\HHE}{{\mathsf H\mathsf E}}

\nc{\he}{{\mathrm{h}}}
\nc{\ph}{{\mathrm{ph}}}
\nc{\ac}{{\mathrm{ac}}}

\theoremstyle{plain}

\newtheorem{theorem}{Theorem}[section]

\newtheorem{lemma}[theorem]{Lemma}
\newtheorem{proposition}[theorem]{Proposition}
\newtheorem{corollary}[theorem]{Corollary}

\theoremstyle{definition}

\newtheorem{definition}[theorem]{Definition}

\theoremstyle{remark}

\newtheorem{remark}[theorem]{Remark}

\title[Height of exceptional collections and Hochschild cohomology]{Height of exceptional collections\\and Hochschild cohomology of quasiphantom categories}
\author{Alexander Kuznetsov}
\address{\sloppy
\parbox{0.9\textwidth}{
Algebraic Geometry Section, Steklov Math Institute,
8 Gubkin str., Moscow 119991 Russia
\hfill\\[5pt]
The Poncelet Laboratory, Independent University of Moscow
\hfill\\[5pt]
Laboratory of Algebraic Geometry, SU-HSE
\hfill
}\bigskip}
\email{akuznet@mi.ras.ru}

\date{}
\thanks{I was partially supported by
RFFI grants 10-01-93110, 10-01-93113, 11-01-00393, 11-01-00568, \hbox{11-01-92613-KO-a}, 12-01-33024, NSh-5139.2012.1,
the grant of the Simons foundation, and by AG Laboratory SU-HSE, RF government grant, ag.11.G34.31.0023.}

\begin{document}

\begin{abstract}
We define the normal Hochschild cohomology of an admissible subcategory of the derived category
of coherent sheaves on a smooth projective variety $X$ --- a graded vector space
which controls the restriction morphism from the Hochschild cohomology of $X$
to the Hochschild cohomology of the orthogonal complement of this admissible subcategory. 
When the subcategory is generated by an exceptional collection, we define its new invariant (the height)
and show that the orthogonal to an exceptional collection of height $h$ in the derived category of a smooth 
projective variety $X$ has the same Hochschild cohomology as $X$ in degrees up to $h - 2$. 
We use this to describe the second Hochschild cohomology of quasiphantom categories 
in the derived categories of some surfaces of general type. We also give necessary 
and sufficient conditions of fullness of an exceptional collection in terms of its 
height and of its normal Hochschild cohomology.
\end{abstract}

\maketitle


\section{Introduction}

Assume $X$ is a smooth projective variety and $\BD^b(\coh(X)) = \langle \CA,\CB \rangle$ is a semiorthogonal
decomposition. The goal of the present paper is to describe the Hochschild cohomology of $\CA$ and the restriction
morphism $\HOH^\bullet(X) \to \HOH^\bullet(\CA)$ in terms of the category $\CB$, especially in case when $\CB$
is generated by an exceptional collection.

This question is motivated by the investigation of so-called quasiphantom categories. A quasiphantom category $\CA$
is a semiorthogonal component of the bounded derived category of coherent sheaves on a smooth projective variety
which has trivial Hochschild homology. Recently, several examples of such categories with $X$ being a surface
of general type have been constructed:
when $X$ is the classical Godeaux surface a quasiphantom was constructed by B\"ohning--Graf von Bothmer--Sosna \cite{BBS},
when $X$ is a Burniat surface --- by Alexeev--Orlov \cite{AO},
when $X$ is the Beauville surface --- by Galkin--Shinder \cite{GS}, and
when $X$ is a determinantal Barlow surface --- by B\"ohning--Graf von Bothmer--Katzarkov--Sosna \cite{BBKS}.

In all these examples the quasiphantom is the orthogonal complement of an exceptional collection.

The structure of quasiphantom categories is very interesting but little understood as yet.
In particular, no direct way to compute their invariants such as Hochschild cohomology is known.
So, it is natural to do this using the information from the given semiorthogonal decomposition.
This leads to the question formulated in the first paragraph. 

The first result of the paper answers this question in the most general form. 
Let $\SCD$ be a smooth and proper DG-category and $\SCB \subset \SCD$ its DG-subcategory.
We define the normal Hochschild cohomology of $\SCB$ in $\SCD$ as the derived tensor product
of DG-bimodules
\begin{equation*}
\NHH^\bullet(\SCB,\SCD) = \SCB \lotimes_{\SCB^\opp\otimes\SCB} \SCD^\vee_\SCB,
\end{equation*}
where the first factor of the above tensor product is the diagonal $\SCB$-bimodule while in the second factor 
$\SCD^\vee = \RHom_{\BD(\SCD^\opp\otimes\SCD)}(\SCD,\SCD\otimes_\kk\SCD)$ is the dual of the diagonal $\SCD$-bimodule and $\SCD^\vee_\SCB$ is its restriction to~$\SCB$. 
We show that if $X$ is a smooth projective variety, $\SCD$ is a pretriangulated enhancement of $\BD^b(\coh(X))$, 
and $\SCB \subset \SCD$ is the induced enhancement of the semiorthogonal component $\CB \subset \BD^b(\coh(X))$ then
there is a distinguished triangle
\begin{equation*}
\NHH^\bullet(\SCB,\SCD) \to \HOH^\bullet(X) \to \HOH^\bullet(\CA).
\end{equation*}
We define the {\em height} of the subcategory $\CB$ as the minimal integer $h$
such that $\NHH^h(\SCB,\SCD) \ne 0$. It follows immediately that for $t \le h-2$
the restriction morphism $\HOH^t(X) \to \HOH^t(\CA)$ is an isomorphism and for $t = h-1$
it is a monomorphism.

Of course, a reasonable way of computing the height (and the normal Hochschild cohomology) is required. 
In case when the subcategory $\CB$ is generated by an exceptional collection $E_1,\dots,E_n$, we construct 
a spectral sequence which computes $\NHH^\bullet(\SCB,\SCD)$ in terms of $\Ext$-groups $\Ext^\bullet(E_i,E_j)$ 
and $\Ext^\bullet(E_i,S^{-1}(E_j))$, where $S^{-1}(F) = F\otimes\omega_X^{-1}[-\dim X]$ is the inverse Serre functor. 
The differentials  in the spectral sequence are expressed in terms of the Yoneda product and higher multiplications
in the natural $A_\infty$ structure. 

Of course, usually it is not easy to control higher multiplications, so explicit computation 
of the height may be difficult. So, we define the {\em pseudoheight} of an exceptional collection
$E_1,\dots,E_n$ as the minimal integer $h$ such that the first page of the above spectral sequence
has a nontrivial term in degree $h$. By definition, the pseudoheight bounds the height from below,
and thus controls the restriction morphism of Hochschild cohomology as well. At the same time the computation
of the pseudoheight does not require any information on (higher) multiplications, and so it is easy manageable. 

We illustrate the computation of the height and of the pseudoheight by considering the quasiphantoms
in the classical Godeaux, Burniat and the Beauville surfaces. We show that in all cases the height
of the collections is $4$, while the pseudoheight varies from $4$ to $3$ depending on the particular case.
It follows that the restriction morphism $\HOH^t(X) \to \HOH^t(\CA)$ is an isomorphism for $t \le 2$
and a monomorphism for $t = 3$ in all these cases. We also deduce from this the fact that the formal
deformation spaces of all considered surfaces $X$ are isomorphic to the formal deformation spaces
of the quasiphantom subcategories.

Finally, we show that the height (and the pseudoheight) can be used to verify whether a given exceptional collection is full.
On one hand, if the height is strictly positive, one can easily deduce that the collection is not full.
On the other hand, we give a sufficient condition of fullness of an exceptional collection which uses
the spectral sequence computing the height of an exceptional collection and seems to be related
to quantum determinants considered by Bondal and Polishchuk.

%


{\bf Acknowledgement.} I would like to thank Dima Orlov for suggesting the problem and useful discussions,
Roman Mikhailov for a help with references, and Olaf Schn\"urer for corrections to the first version of the paper.
I am very grateful to Christian B\"ohning, Hans-Christian Graf von Bothmer and Pawel Sosna for the help
with computing the height of their exceptional collection, careful reading of the first draft of the paper,
and many useful comments.

\section{Preliminaries}

\subsection{DG-categories}

For a detailed survey of DG-categories we refer to~\cite{Ke} and references therein. 
Here we only sketch some basic definitions.

A {\sf DG-category} over a field $\kk$ is a category $\SCD$ such that for any pair of objects $\bx,\by \in \SCD$ 
the set of morphisms $\Hom_\SCD(\bx,\by)$ comes with a structure of a chain complex of $\kk$-vector spaces such that for each object $\bx$ 
the identity morphism $\id_\bx \in \Hom_\SCD(\bx,\bx)$ is closed of degree zero and the composition map 
$\Hom_\SCD(\bx,\by)\otimes\Hom_\SCD(\by,\bz) \to \Hom_\SCD(\bx,\bz)$ is a morphism of complexes (the Leibniz rule). 
The simplest example of a DG-category is the category $\dgm\kk$ of complexes
of $\kk$-vector spaces with
\begin{equation*}
\Hom_{\dgm\kk}(M^\bullet,N^\bullet)^t = \prod_{i \in \ZZ}\Hom(M^i,N^{i+t}),
\qquad
d^t(f^i) = (d^{i+t}_N\circ f^i - (-1)^t f^{i+1}\circ d^i_M).
\end{equation*}

The {\sf homotopy category} $[\SCD]$ of a DG-category $\SCD$ is defined
as the category which has the same objects as $\SCD$ and with
\begin{equation*}
\Hom_{[\SCD]}(\bx,\by) = H^0(\Hom_\SCD(\bx,\by)).
\end{equation*}
For example, the homotopy category $[\dgm\kk]$ is the category of complexes with morphisms being chain morphisms of complexes up to a homotopy.

A {\sf DG-functor} $F:\SCD \to \SCD'$ is a $\kk$-linear functor such that for any pair of objects $\bx,\by \in \SCD$ 
the morphism $F:\Hom_\SCD(\bx,\by) \to \Hom_{\SCD'}(F(\bx),F(\by))$ is a morphism of complexes. 
If $\SCD$ is a small DG-category (objects form a set) then DG-functors from $\SCD$ to $\SCD'$ also form a DG-category with
\begin{equation*}
\Hom(F,G) = \Ker\left( \prod_{\bx\in\SCD}\Hom_{\SCD'}(F(\bx),G(\bx)) \to \prod_{\bx,\by \in \SCD}\Hom(\Hom_\SCD(\bx,\by),\Hom_{\SCD'}(F(\by),G(\bx))) \right)
\end{equation*}

A {\sf right DG-module} over $\SCD$ is a DG-functor from $\SCD^\opp$, the opposite DG-category, to $\dgm\kk$. 
A {\sf left DG-module} over $\SCD$ is a DG-functor from $\SCD$ to $\dgm\kk$. The DG-category of right (resp.\ left)
DG-modules over $\SCD$ is denoted by $\dgm\SCD$ (resp.\ $\dgm\SCD^\opp$). The homotopy category $[\dgm\SCD]$ of DG-modules
has a natural triangulated structure. Moreover, it has arbitrary direct sums.

The {\sf Yoneda DG-functor} $h:\SCD \to \dgm\SCD$ is defined by 
\begin{equation*}
\bx \mapsto h^\bx(\by) = \Hom_\SCD(\by,\bx).
\end{equation*}
The Yoneda DG-functor for the opposite DG-category can be written as $\bx \mapsto h_\bx(\by) = \Hom_\SCD(\bx,\by)$.
The Yoneda DG-functors are full and faithful. Moreover, one has
\begin{equation*}
\Hom_{\dgm\SCD}(h^\bx,M) = M(\bx),\qquad
\Hom_{\dgm\SCD^\opp}(h_\bx,N) = N(\bx)
\end{equation*}
for any right $\SCD$-module $M$ and any left $\SCD$-module $N$.
The DG-modules in the images of the Yoneda functors are called {\sf representable}. 
The minimal triangulated subcategory of $[\dgm\SCD]$ containing all representable DG-modules and closed under direct summands
is called {\sf the category of perfect DG-modules over $\SCD$} and is denoted by~$\Perf(\SCD)$. Its objects are called 
{\sf perfect} DG-modules.

A DG-module $M$ is {\sf acyclic} if for each object $\bx\in\SCD$ the complex $M(\bx) \in \dgm\kk$ is acyclic. 
The DG-category of acyclic DG-modules is denoted by $\acy\SCD$. The {\sf derived category of a DG-category $\SCD$} 
is defined as the Verdier quotient
\begin{equation*}
\BD(\SCD) = [\dgm\SCD]/[\acy\SCD].
\end{equation*}
It has a natural triangulated structure. The quotient functor $[\dgm\SCD] \to \BD(\SCD)$ commutes with arbitrary
direct sums and its restriction onto the category of perfect DG-modules is fully faithful, $\Perf(\SCD) \subset \BD(\SCD)$.
In fact, the category of perfect DG-modules identifies with the subcategory $\BD(\SCD)^\comp$ of {\em compact objects} in $\BD(\SCD)$
(recall that an object $\bx$ of a category is {\sf compact} if the functor $\Hom(\bx,-)$ commutes with arbitrary direct sums).
A DG-category $\SCD$ is called {\sf pretriangulated} if $[\SCD] \subset \BD(\SCD)$ is a triangulated subcategory.

Sometimes it is convenient to have a description of $\BD(\SCD)$ not using the Verdier quotient construction.
One way is to consider the subcategory of $\dgm\SCD$ consisting of all DG-modules $P$ such that for any
acyclic DG-module $A$ the complex $\Hom_{\dgm\SCD}(P,A)$ is acyclic. Such DG-modules $P$ are called
{\sf homotopically projective}, or simply {\sf h-projective}. The full subcategory of $\dgm\SCD$ consisting
of h-projective DG-modules is denoted by $\hproj\SCD$. Its homotopy category is equivalent to the derived 
category 
\begin{equation*}
[\hproj\SCD] \cong \BD(\SCD).
\end{equation*}
It is easy to see that any perfect DG-module is h-projective.

Let $\CT$ be a triangulated category. An enhancement for $\CT$ is a choice of a pretriangulated DG-category $\SCD$ and 
of an equivalence $\epsilon:[\SCD] \to \CT$ of triangulated categories.


\subsection{DG-bimodules and tensor products}

If $\SCD_1$ and $\SCD_2$ are DG-categories over $\kk$, the {\sf tensor product} $\SCD_1\otimes_\kk\SCD_2$ is the DG-category
whose objects are pairs $(\bx_1,\bx_2)$ with $\bx_i$ being objects of $\SCD_i$, and with morphisms
defined by 
\begin{equation*}
\Hom_{\SCD_1\otimes\SCD_2}((\bx_1,\bx_2),(\by_1,\by_2)) = \Hom_{\SCD_1}(\bx_1,\by_1)\otimes_\kk \Hom_{\SCD_2}(\bx_2,\by_2).
\end{equation*}
A {\sf $\SCD_1$-$\SCD_2$ DG-bimodule} is a DG-module over $\SCD_1^\opp\otimes\SCD_2$. In other words,
a DG-bimodule $\varphi$ associates with any pair of objects $\bx_1 \in \SCD_1$, $\bx_2 \in \SCD_2$
a complex $\varphi(\bx_1,\bx_2)$ and a collection of morphisms of complexes
\begin{equation*}
\Hom_{\SCD_1}(\bx_1,\by_1) \otimes \varphi(\bx_1,\bx_2) \to \varphi(\by_1,\bx_2),\qquad
\varphi(\bx_1,\bx_2)\otimes \Hom_{\SCD_2}(\by_2,\bx_2) \to \varphi(\bx_1,\by_2)
\end{equation*}
for all $\by_1\in\SCD_1$, $\by_2\in\SCD_2$, which commute and are compatible with the composition laws
in $\SCD_1$ and $\SCD_2$. One of the most important examples of a DG-bimodule
is the {\sf diagonal $\SCD$-$\SCD$ DG-bimodule} $\SCD$ defined by
\begin{equation*}
\SCD(\bx_1,\bx_2) = \Hom_\SCD(\bx_2,\bx_1).
\end{equation*}
Other examples can be constructed as {\sf exterior products} of left $\SCD_1$-modules $M$ with right $\SCD_2$-modules $N$ 
\begin{equation*}
(M \otimes_\kk N)(\bx_1,\bx_2) = M(\bx_1)\otimes_\kk N(\bx_2).
\end{equation*}

Let $\varphi$ be a $\SCD_1$-$\SCD_2$ DG-bimodule and $\psi$ a $\SCD_2$-$\SCD_3$ DG-bimodule.
Their {\sf tensor product over $\SCD_2$} is a $\SCD_1$-$\SCD_3$ DG-bimodule defined by
\begin{multline*}
(\varphi\otimes_{\SCD_2}\psi)(\bx_1,\bx_3) := \\
\Coker \left(
\bigoplus_{\bx_2,\by_2\in\SCD_2} \varphi(\bx_1,\bx_2)\otimes_\kk \Hom_{\SCD_2}(\by_2,\bx_2) \otimes_\kk \psi(\by_2,\bx_3) \to
\bigoplus_{\bx_2\in\SCD_2} \varphi(\bx_1,\bx_2)\otimes_\kk \psi(\bx_2,\bx_3) 
\right).
\end{multline*}
A straightforward verification shows that 
\begin{equation*}
\varphi\otimes_{\SCD_2}\SCD_2 = \varphi
\qquad\text{and}\qquad
\SCD_2\otimes_{\SCD_2}\psi = \psi.
\end{equation*}
Of course, taking either $\SCD_1$ or $\SCD_3$ to be just the base field we obtain the tensor product
in the appropriate categories of DG-modules, like $-\otimes_{\SCD_2}-:\dgm\SCD_2 \times \dgm\SCD_2^\opp \to \dgm\kk$.
The same verification as above shows that
\begin{equation*}
\varphi\otimes_{\SCD_2}h_{\bx_2} = \varphi(-,\bx_2)
\qquad\text{and}\qquad
h^{\bx_1}\otimes_{\SCD_1}\varphi = \varphi(\bx_1,-)
\end{equation*}
for any $\SCD_1$-$\SCD_2$-bimodule $\varphi$ and any objects $\bx_1 \in \SCD_1$, $\bx_2 \in \SCD_2$.

The {\sf derived tensor product} is defined by replacing etiher of the factors by an h-projective resolution
\begin{equation*}
\varphi \lotimes_{\SCD_2} \psi := P(\varphi) \otimes_{\SCD_2} \psi \cong \varphi \otimes_{\SCD_2} P(\psi).
\end{equation*}
It is a bifunctor on derived categories  $\BD(\SCD_1^\opp\otimes\SCD_2) \times \BD(\SCD_2^\opp \otimes \SCD_3) \to \BD(\SCD_1^\opp \otimes \SCD_3)$.
Since representable modules are h-projective, we have 
\begin{equation*}
\varphi\lotimes_{\SCD_2}h_{\bx_2} = \varphi(-,\bx_2)
\qquad\text{and}\qquad
h^{\bx_1}\lotimes_{\SCD_1}\varphi = \varphi(\bx_1,-)
\end{equation*}
Also, there is a nice choice of an h-projective resolution for the diagonal bimodule, called the {\sf bar-resolution}.
It is defined by
\begin{equation}\label{bard}
\BB(\SCD) = \bigoplus_{p=0}^\infty \quad \bigoplus_{\bx_0,\dots,\bx_p \in \SCD} 
h_{\bx_p} \otimes_\kk \SCD(\bx_p,\bx_{p-1}) \otimes_\kk \dots \otimes_\kk \SCD(\bx_1,\bx_0) \otimes_\kk h^{\bx_0}[p]
\end{equation}
with the differential consisting of the differentials of $h_{\bx_p}$, $h^{\bx_0}$, and $\SCD(\bx_i,\bx_{i-1})$
and of the compositions $h_{\bx_p}\otimes\SCD(\bx_p,\bx_{p-1}) \to h_{\bx_{p-1}}$, 
$\SCD(\bx_{i+1},\bx_i) \otimes \SCD(\bx_i,\bx_{i-1}) \to \SCD(\bx_{i+1},\bx_{i-1})$
and $\SCD(\bx_1,\bx_0)\otimes h^{\bx_0} \to h^{\bx_1}$. Using this resolution it is easy
to see that 
\begin{equation*}
\varphi\lotimes_{\SCD_2}\SCD_2 \cong \varphi
\qquad\text{and}\qquad
\SCD_2\lotimes_{\SCD_2}\psi \cong \psi.
\end{equation*}

Given a DG-functor $F:\SCD_1 \to \SCD_2$ we define the restriction and the derived induction functors by
\begin{equation*}
\begin{array}{rll}
\Res_F:&\BD(\SCD_2) \to \BD(\SCD_1), \qquad & \Res_F(M)(\bx) = M(F(\bx)),\\
\LInd_F:&\BD(\SCD_1) \to  \BD(\SCD_2), & \LInd_F(N) = N\lotimes_{\SCD_1}{}_F\SCD_2,
\end{array}
\end{equation*}
where ${}_F\SCD_2(\bx,\by) = \SCD_2(F(\bx),\by)$ is the restriction of the diagonal bimodule $\SCD_2$ via the first argument.
Both functors are triangulated and commute with arbitrary direct sums. The derived induction functor takes representable
DG-modules to representable DG-modules. The derived induction functor is left adjoint to the restriction functor. 
Moreover, if $F$ is a quasiequivalence then $\LInd_F$ and $\Res_F$ are mutually inverse equivalences of the derived categories.


\subsection{Cosimplicial machinery}

In the next subsection we will construct a DG-enhancement of the derived category of coherent sheaves
on a smooth projective scheme $X$, which is based on the \v{C}ech complex. Since the \v{C}ech complex
is coismplicial by its nature, we will need some facts about the category of cosimplicial complexes
which we recall here. The simplicial versions of these results are very well known and can be found
in many textbooks, see e.g.~\cite{We}. For the cosimplicial version we refer to the Appendix in~\cite{GM}.


For any category $\SCA$ a {\sf cosimplicial object in $\SCA$} is just a functor $\Delta \to \SCA$, where $\Delta$
is the category of finite nonempty linearly ordered sets. We denote by $[p] := \{0,1,\dots,p\} \in \Delta$, $p \ge 0$, 
(with the natural order) the objects of $\Delta$. Then a cosimplicial object $A:\Delta \to \SCA$ is given by
a collection of objects $A^p := A([p]) \in \SCA$ and a collection of morphisms $f^*:A^p \to A^q$ for any nondecreasing 
map $f:[p] \to [q]$. The unique strictly increasing map $[p-1] \to [p]$ with the set of values $[p] \setminus \{i\}$
(as well as the induced morphism $A^{p-1} \to A^p$) is called the $i$-th {\sf face map} and is denoted by $\partial^i_p$.
Analogously, the unique surjective nondecreasing map $[p+1] \to [p]$ which takes value $i$ twice is called
the {\sf $i$-th degeneration map} and is denoted by $s^i_p$.

Assume $\SCA$ is abelian. Given a cosimplicial object $A^\bullet$ in $\SCA$ one defines the corresponding 
{\sf normalized complex} by
\begin{equation*}
(NA)^p = 
\bigcap_{i=0}^{p-1} \Ker\left(A^{p} \xrightarrow{\ s^i_{p-1}\ } A^{p-1}\right),
\qquad
d_{NA}^p = \sum_{i=0}^{p+1}(-1)^i\partial^i_{p+1}.
\end{equation*}

Assume now that $\SCA$ is a monoidal category. Then the functor $N$ is a lax monoidal functor.
It means that for any cosimplicial objects $A^\bullet,B^\bullet$ in $\SCA$ there is a morphism of functors
\begin{equation*}
\Delta_{A,B}:N(A) \otimes N(B) \to  N(A\otimes B)
\end{equation*}
called the {\sf Alexander--Whitney map} and defined by
\begin{equation*}
\Delta_{A,B}(a\otimes b) = (\partial_{p+q}^{p+q}\circ\dots\circ\partial_{p+1}^{p+1})(a) \otimes  (\partial^{0}_{p+q}\circ\dots\circ\partial^{0}_{q+1})(b),
\qquad
a \in A^p,\ b \in B^q,
\end{equation*}
as well as a morphism of functors
\begin{equation*}
\nabla_{A,B}:N(A\otimes B) \to N(A) \otimes N(B),
\end{equation*}
called the {\sf Eilenberg--Zilber map} and defined by
\begin{equation*}
\nabla_{A,B}(a\otimes b) = \sum_{p+q=n} \sum_{\tau \in \Shuf(p,q)} (-1)^{|\tau|} 
(s_p^{\tau_{p+1}-1}\circ\dots\circ s^{\tau_{p+q}-1}_{p+q-1})(a)\otimes 
(s_1^{\tau_{1}-1}\circ\dots\circ s^{\tau_{p}-1}_{p+q-1})(b),
\quad
a \in A^n,\ b \in B^n.
\end{equation*}
Here $\Shuf(p,q)$ is the set of all $(p+q)$ permutations $\tau$ such that $\tau_1<\dots<\tau_p$ and $\tau_{p+1} < \dots < \tau_{p+q}$
(such permutations are called {\sf $(p,q)$-shuffles}), and $|\tau|$ is the parity of the permutation $\tau$.
One can check that
\begin{equation*}
\nabla_{A,B}\circ \Delta_{A,B} = \id,
\qquad
\Delta_{A,B}\circ \nabla_{A,B} \sim \id,
\end{equation*}
where $\sim$ means homotopic.
In particular, both maps are quasiisomorphisms.

Finally, we will need the following construction, called the {\sf edgewise subdivision} operation, 
which is well known to experts, but we have not found a good reference for it. It looks like it was 
introduced by Segal~\cite{Se} in the context of simplicial spaces, but Segal attributes it to Quillen. 
Below is a (slightly modified) cosimplicial version of this construction.

Consider the endofunctor $\sd:\Delta \to \Delta$ which takes an object $[p]$ to the object $[2p+1]$
and a morphism $f:[p] \to [q]$ to the morphism $\sd(f):[2p+1] \to [2q+1]$ defined by
\begin{equation*}
\sd(f)(i) = 
\begin{cases}
f(i), & \text{if $0 \le i \le p$},\\
q+1+f(i-p-1),& \text{if $p+1 \le i \le 2p+1$}
\end{cases}
\end{equation*}
In other words, the functor $\sd$ takes a linearly ordered set into the union of two such sets (ordered so that 
the second copy goes after the first), and a map $f$ to the map coinciding with $f$ on each copy.
This is precisely written in the above formula. 

The functor $\sd$ comes with a morphism of functors $\sigma:\id_\Delta \to \sd$, which maps 
$[p]$ to $[2p+1]$ by the isomorphism onto the first copy of $[p]$, i.e.\ $\sigma(i) = i$.
Commutation of $\sigma$ with the faces and degenerations is evident. 

The subdivision functor $\sd$ acts on cosimplicial objects by taking $A^\bullet:\Delta \to \SCA$
to $A^\bullet\circ\sigma:\Delta \to \SCA$ and thus induces an endofunctor on the category
of cosimplicial objects which we also denote by $\sd$. Explicitly
\begin{equation*}
\sd(A^\bullet) = A^{2\bullet +1},
\qquad
\sd(\partial^i_p) = \partial_{2p+1}^{i+p+1}\circ\partial_{2p}^{i},\quad 
\sd(s_p^i) = s_{2p+1}^{i+p+1}\circ s_{2p+2}^i.
\end{equation*}
It is clear that $\sigma$ induces a morphism of cosimplicial objects $A^\bullet \to \sd(A^\bullet)$
and hence a morphism of their normalizations
\begin{equation*}
N(A) \xrightarrow[\cong]{\ \sigma\ } N(\sd(A)).
\end{equation*}
It is known that this map $\sigma$ is a quasiisomorphism of complexes (Segal proves in~\cite[Prop.~A.1]{Se} that
it induces an isomorphism of geometric realizations which is a stronger statement).


\subsection{Enhancements for derived categories of coherent sheaves}

Let $X$ be a smooth projective variety. Denote by $\BD^b(\coh(X))$ the derived category of coherent sheaves on $X$.
In this section we construct an enhancement for it.

First we construct a cosimplicial enhancement.
Choose a finite affine covering $\{\CU_\alpha\}_{\alpha \in A}$ for~$X$ with $A$ being the index set.
It gives a simplicial object $\FU_\bullet$ in the category $\Open(X)$ of open coverings of $X$ defined by
\begin{equation*}
\FU_p = \bigsqcup_{\alpha:[p] \to A} \CU_{\alpha_0}\cap \dots \cap\CU_{\alpha_p},
\end{equation*}
the union is over the set of all maps from $[p]$ to $A$ with no conditions on the ordering (in fact we do not 
have an order on $A$ so such a condition is impossible to write down).
Each complex of quasicoherent sheaves $G$ on $X$ gives a contravariant functor from the category
of open coverings $\Open(X)$ to the category of complexes of vector spaces. Applying it to the simplicial
covering $\FU_\bullet$ we obtain a cosimplicial complex of vector spaces 
%
%
\begin{equation}\label{bc}
\BC^p(X,G) := G(\FU_\bullet) = \bigoplus_{(\alpha_0,\dots,\alpha_p) \in A^{p+1}}
\Gamma(\CU_{\alpha_0}\cap \dots \cap\CU_{\alpha_p},G)
\end{equation}
The totalization (via direct sums) of its normalization 
$N\BC^\bullet(X,G)$ is the \v{C}ech complex of~$G$. In particular, the cohomology
of $N\BC^\bullet(X,G)$ are isomorphic to the hypercohomology of $G$.

Let $\SCC = \SCC_X$ be the cosimplicial category with objects being finite complexes of vector bundles on $X$ and with
\begin{equation}\label{cfg}
\Hom^\bullet_{\SCC}(F,G) =  \BC^\bullet(X,F^\vee\otimes G).
\end{equation}
It is easy to see that $\SCC_X$ is a cosimplicial category.
We call it {\sf \v{C}ech cosimplicial category of $X$}.

On the other hand, for any cosimplicial category $\SCC$ one can define a DG-category $N\SCC$,
{\sf the normalization of $\SCC$}, with the same objects as in $\SCC$ and with
\begin{equation*}
\Hom_{N\SCC}(\bx,\by) = N(\Hom^\bullet_{\SCC}(\bx,\by)).
\end{equation*}
The composition law in $N\SCC$ is defined via the Alexander--Whitney map
\begin{multline*}
\Hom_{N\SCC}(\bx,\by)\otimes \Hom_{N\SCC}(\by,\bz) =
N(\Hom_{\SCC}(\bx,\by)) \otimes N(\Hom_{\SCC}(\bx,\by)) \xrightarrow{\ \Delta\ }
\\  \xrightarrow{\ \Delta\ }
N(\Hom_{\SCC}(\bx,\by) \otimes \Hom_{\SCC}(\bx,\by)) \xrightarrow{\ m_\SCC\ }
N(\Hom_{\SCC}(\bx,\bx)) =
\Hom_{N\SCC}(\bx,\bz),
\end{multline*}
where $m_\SCC$ is the composition law in $\SCC$. It is an exercise to check that $N\SCC$ is then a DG-category.

We define a DG-category $\SCD_X$ as the normalization of the \v{C}ech cosimplicial category
\begin{equation*}
\SCD_X := N\SCC_X.
\end{equation*}
It will be referred to as {\sf \v{C}ech DG-category of $X$}.


It is clear that $[\SCD_X] \cong \BD^b(\coh(X))$. Indeed, the categories have the same objects and the spaces 
of morphisms are also the same (because we can compute the cohomology by \v{C}ech complex). We denote the equivalence 
by $\epsilon:[\SCD_X] \to \BD^b(\coh(X))$. Thus the DG-category $\SCD_X$ provides an enhancement of $\BD^b(\coh(X))$.
We will call it {\sf the \v{C}ech enhancement with respect to the covering $\{\CU_\alpha\}$}. 

Note that equivalence $\epsilon$ extends to an equivalence $\BD(\SCD_X) \cong \BD(X)$ between the derived category
of $\SCD_X$ and the unbounded derived category of quasicoherent sheaves on $X$.
%
Indeed, for each complex of quasicoherent sheaves $G$ on $X$ we define a DG-module $\bar\epsilon(G)$ over $\SCD_X$ by
\begin{equation*}
\bar\epsilon(G)(F) := N\BC^\bullet(X,F^\vee\otimes G),
\qquad F \in \SCD_X,
\end{equation*}
It is easy to see that this is a DG-functor $\bar\epsilon:\com(\Qcoh(X)) \to \dgm\SCD_X$ which takes acyclic complexes
to acyclic DG-modules and so inducing a functor $\bar\epsilon:\BD(X) \to \BD(\SCD_X)$.
It is also easy to check that $\bar\epsilon$ is an equivalence such that the diagram
\begin{equation*}
\xymatrix{
[\SCD_X] \ar[d]_h \ar[r]^-\epsilon & \BD^b(\coh(X)) \ar@{^{(}->}[d] \\
\BD(\SCD_X) & \BD(X) \ar[l]_-{\bar\epsilon}
}
\end{equation*}
commutes, where the left vertical arrow is the Yoneda embedding and the right vertical arrow is the canonical embedding.

%

One nice property of the \v{C}ech enhancement is that 
\begin{equation*}
\Hom_{\SCD_X}(F^\vee,G^\vee) = \Hom_{\SCD_X}(G,F).
\end{equation*}
This means that we have in fact two enhancements
\begin{equation*}
\epsilon:[\SCD_X] \xrightarrow{\ \cong\ } \BD^b(\coh(X))
\qquad\text{and}\qquad
\epsilon^\vee:[\SCD_X^\opp] \xrightarrow{\ \cong\ } \BD^b(\coh(X))
\qquad\text{with}\qquad
\epsilon^\vee(\bx) = \epsilon(\bx)^\vee.
\end{equation*}
Moreover, for any perfect left $\SCD_X$-module $M$ and any perfect right $\SCD_X$-module $N$ we have
\begin{equation}\label{mdn}
M\lotimes_{\SCD_X} N \cong \HH^\bullet(X,\epsilon(M)^\vee\lotimes\epsilon(N)),
\end{equation}
the tensor product in the RHS is over $\CO_X$.
Indeed, $\SCD_X$ is pretriangulated and $[\SCD(X)] = \BD^b(\coh(X))$ is Karoubian closed, hence all perfect
DG-modules over $\SCD_X$ are representable, so we may assume that $M =h_\bx$ and $N = h^\by$ are both representable.
the LHS gives $\HH^\bullet(\Hom_{\SCD_X}(\bx,\by))$ which can be rewritten as
$\Ext^\bullet(\epsilon(\bx),\epsilon(\by)) \cong 
\HH^\bullet(X,\epsilon(\bx)^\vee\lotimes\epsilon(\by))$
and this coincides with the RHS. 

Now consider the square $X\times X$ and its open affine covering $\FU^2_\bullet$ by sets $\{\CU_\alpha \times \CU_\beta\}$.
It is easy to see that for any exterior product of bounded complexes $G_1\boxtimes G_2$ of coherent sheaves
on the factors we have
\begin{equation*}
\BC^\bullet(X\times X,G_1\boxtimes G_2) = \BC^\bullet(X,G_1) \otimes \BC^\bullet(X,G_2),
\end{equation*}
the tensor product of cosimplicial complexes.

It follows that for any bounded complexes of vector bundles $F_1,F_2,G_1,G_2$ on $X$ we have
\begin{multline*}
\Hom_{\SCC_{X\times X}}(F_2\boxtimes F_1^\vee,G_2\boxtimes G_1^\vee) = 
\BC^\bullet(X\times X,(F_2^\vee\otimes G_2) \boxtimes (F_1\otimes G_1^\vee)) = \\ =
\BC^\bullet(X,F_2^\vee\otimes G_2) \otimes \BC^\bullet(X,F_1\otimes G_1^\vee) = 
\Hom_{\SCC_X^\opp}(F_1,G_1) \otimes \Hom_{\SCC_X}(F_2,G_2).
\end{multline*}
where $\SCC_{X\times X}$ is the \v{C}ech cosimplicial category of $X$.
Moreover, this isomorphism is compatible with the composition laws.
In other words, it gives a fully faithful cosimplicial functor
\begin{equation}\label{vka}
\kappa:\SCC_X^\opp\otimes\SCC_X \to \SCC_{X\times X},\qquad
(F_1,F_2) \mapsto F_2\boxtimes F_1^\vee,
\end{equation}
which we call the {\sf K\"unneth functor}. 
Combining with the normalization we get a fully faithful DG-functor
\begin{equation*}
\kappa:N(\SCC_X^\opp\otimes\SCC_X) \to N\SCC_{X\times X} = \SCD_{X\times X}.
\end{equation*}
On the other hand, the Eilenberg--Zilber map gives a DG-functor
\begin{equation*}
N(\SCC_X^\opp\otimes\SCC_X) \xrightarrow{\ \nabla\ } N\SCC_X^\opp \otimes N\SCC_X = \SCD_X^\opp \otimes \SCD_X
\end{equation*}
(this also can be easily varified), which is a quasiequivalence. Thus we have a diagram
\begin{equation*}
\xymatrix{
& N(\SCC_X^\opp\otimes\SCC_X) \ar[dl]_\nabla \ar[dr]^{\kappa} \\
\SCD_X^\opp \otimes \SCD_X  && \SCD_{X\times X}
}
\end{equation*}
This allows to define a functor $\BD(\SCD_X^\opp\otimes\SCD_X) \to \BD(\SCD_{X\times X})$ by composing the restriction 
with respect to $\nabla$ with the derived induction along $\kappa$. Composing it further with the equivalence quasiinverse to
$\bar\epsilon_{X\times X}:\BD(X\times X) \cong \BD(\SCD_{X\times X})$ we obtain a functor 
$\mu:\BD((\SCD_X^\opp\otimes\SCD_X) \to \BD(X\times X)$ such that the diagram
\begin{equation}\label{dmu}
\vcenter{\xymatrix{
\BD(\SCD_X^\opp\otimes\SCD_X) \ar[rr]^-{\Res_\nabla} \ar[d]_\mu && \BD(N(\SCC_X^\opp\otimes\SCC_X)) \ar[d]^{\LInd_\kappa} \\
\BD(X\times X) \ar[rr]^-{\bar\epsilon_{X\times X}} && \BD(\SCD_{X\times X})
}}
\end{equation}
commutes.
%
%

Note that since the DG-functor $\nabla$ is a quasiequivalence, the functor $\Res_\nabla$ takes a representable bimodule $h_\bx\otimes h^\by$ 
to the representable DG-module $h^{(\bx,\by)}$ over $N(\SCC_X^\opp\otimes\SCC_X)$ (recall that the objects of $\SCD_X^\opp\otimes\SCD_X$
and $N(\SCC_X^\opp\otimes\SCC_X)$ are canonically identified), and the functor $\LInd_\kappa$ takes it further to the representable
DG-module $h^{\kappa(\bx,\by)}$ over $\SCD_{X\times X}$. Comparing with the definition~\eqref{vka} of $\kappa$ we conclude that
\begin{equation}\label{lmu}
\mu(h_\bx\otimes_\kk h^\by) = \epsilon(\by)\boxtimes\epsilon(\bx)^\vee.
\end{equation}
It follows that $\mu$ is an equivalence. Indeed, $\Res_\nabla$ is fully faithful since $\nabla$ is a quasiequivalence,
and $\LInd_\kappa$ is fully faithful since $\kappa$ is. So, since $\bar\epsilon_{X\times X}$ is an equivalence
we conclude that $\mu$ is fully faithful. On the other hand, $\mu$ commutes with arbitrary direct sums, and its
image contains all objects of the form $F_2\boxtimes F_1^\vee$, where $F_1,F_2 \in \BD^b(\coh(X))$ are arbitrary.
Such objects generate $\BD(X\times X)$ hence $\mu$ is essentially surjective.
Moreover, it follows also from~\eqref{lmu} that $\mu$ gives an equivalence
\begin{equation*}
\Perf(\SCD_X^\opp\otimes\SCD_X) \xrightarrow[\cong]{\ \mu\ } \BD^b(\coh(X\times X)).
\end{equation*}


Let us show that $\mu$ is compatible with tensor products. Explicitly, let us show that
\begin{equation}\label{mukk}
\varphi \lotimes_{\SCD_X^\opp\otimes\SCD_X} \varphi' \cong \HH^\bullet(X\times X,\mu(\varphi)\lotimes \mu(\varphi')^T)
\end{equation}
where $\varphi$ and $\varphi'$ are arbitrary DG-bimodules, and
$T$ stands for the transposition (the pullback under the permutation of factors involution of $X\times X$).
%
%
Indeed, if $\varphi = h_\bx\otimes h^\by$ and $\varphi' = h_{\bx'}\otimes h^{\by'}$ then the formula is correct
\begin{multline*}
\HH^\bullet(X\times X,\mu(\varphi)\lotimes \mu(\varphi')^T) \cong 
\HH^\bullet(X\times X,(\epsilon(\by)\boxtimes\epsilon(\bx)^\vee) \lotimes (\epsilon(\bx')^\vee\boxtimes\epsilon(\by'))) \cong \\ \cong
\HH^\bullet(X,\epsilon(\by) \lotimes \epsilon(\bx')^\vee) \otimes \HH^\bullet(X,\epsilon(\bx)^\vee\lotimes\epsilon(\by')) \cong 
(h_{\bx'}\lotimes_\SCD h^\by) \otimes_\kk (h_\bx \lotimes_\SCD h^{\by'}) \cong 
\varphi \lotimes_{\SCD^\opp\otimes\SCD} \varphi'
\end{multline*}
(the first is~\eqref{lmu}, the second is the K\"unneth formula, the third is~\eqref{mdn}, and the fourth is clear).
For arbitrary DG-bimodules the statement follows by devissage. 

As a final preparation step we are going to look on how the pushforward morphism under the diagonal 
embedding $\Delta:X \to X\times X$ acts. Here the edgewise subdivision plays an important role. 
The first observation is that for any complex $G$ we have
\begin{equation}\label{cedi}
\BC^\bullet(X\times X,\Delta_*G) = \sd(\BC^\bullet(X,G)).
\end{equation}
To see this first note that the simplicial covering $\FU^2_\bullet$ of $X\times X$
restricted to the diagonal $X \subset X\times X$ equals $\sd(\FU_\bullet)$, the subdivision
of the initial covering of $X$. Indeed, the index set for the covering of $X\times X$ is $A\times A$
with $\CU_{\alpha',\alpha''} = \CU_{\alpha'}\times\CU_{\alpha''}$. Consequently
\begin{multline*}
\Delta^{-1}(\FU^2_p) =
\Delta^{-1}(\sqcup_{\alpha:[p] \to A\times A}\CU_{\alpha'_0,\alpha''_0}\cap \dots \cap \CU_{\alpha'_p,\alpha''_p}) = \\
\Delta^{-1}(\sqcup_{\alpha',\alpha'':[p] \to A}(\CU_{\alpha'_0}\times\CU_{\alpha''_0})\cap \dots \cap (\CU_{\alpha'_p}\times\CU_{\alpha''_p})) = 
\sqcup_{\alpha',\alpha'':[p] \to A}(\CU_{\alpha'_0}\cup\CU_{\alpha''_0})\cap \dots \cap (\CU_{\alpha'_p}\cup\CU_{\alpha''_p})) = \\
\sqcup_{\alpha',\alpha'':[p] \to A}(\CU_{\alpha'_0}\cup\dots\cup\CU_{\alpha'_p}) \cap (\CU_{\alpha''_0}\cup\dots\cup\CU_{\alpha''_p}) = \sd(\FU_\bullet)_p.
\end{multline*}
Since $\Gamma(U,\Delta_*G) = \Gamma(\Delta^{-1}(U),G)$ for any open subset $U$ of $X\times X$, we deduce~\eqref{cedi}.

Let us look at the image of the diagonal bimodule $\SCD_X$ under this equivalence.
Let $\Delta:X \to X\times X$ be the diagonal embedding.

\begin{lemma}\label{mud}
We have $\mu(\SCD_X) \cong \Delta_*\CO_X$.
\end{lemma}
\begin{proof}
Since $\LInd_\kappa$ is an equivalence of categories and its quasiinverse is given by $\Res_\kappa$, 
it follows from~\eqref{dmu} that it is enough to check that
$\Res_\kappa(\bar\epsilon_{X\times X}(\Delta_*\CO_X)) \cong \Res_\nabla(\SCD_X)$.
This follows from the following chain of quasiisomorphisms of $N(\SCC_X^\opp\otimes\SCC_X)$-modules
%
%
\begin{multline*}
\Res_\kappa(\bar\epsilon_{X\times X}(\Delta_*\CO_X))(\bx,\by) =
N\Hom_{\SCC_{X\times X}}(\kappa(\bx,\by),\Delta_*\CO_X) =
\\ =
N\BC^\bullet(X\times X,(\epsilon(\by)^\vee \boxtimes \epsilon(\bx)) \otimes \Delta_*\CO_X) \cong  
N\BC^\bullet(X\times X,\Delta_*\Delta^*(\epsilon(\by)^\vee \boxtimes \epsilon(\bx))) \cong  
\\ \cong
N(\sd(\BC^\bullet(X,\Delta^*(\epsilon(\by)^\vee \boxtimes \epsilon(\bx)))) \cong 
N\BC^\bullet(X,\Delta^*(\epsilon(\by)^\vee \boxtimes \epsilon(\bx))) \cong  
\\ \cong
N\BC^\bullet(X,\epsilon(\by)^\vee \otimes \epsilon(\bx)) \cong  
N\Hom_{\SCC_X}(\by,\bx) \cong
\Res_\nabla(\SCD_X)(\bx,\by).
\end{multline*}
Here the first equality is the definition of the DG-functor $\bar\epsilon$ and of the functor $\Res$,
the second is the definition of the DG-functor $\kappa$ and of the category $\SCC_{X\times X}$,
the third is the projection formula, the fourth is~\eqref{cedi}, the fifth is the quasiisomorphism $\sigma$,
the sixth is evident, the seventh is the definition of $\SCC_X$, and the eighth is the definition of the diagonal
bimodule $\SCD_X$.
%
\end{proof}

On the other hand, consider the $\SCD_X$-bimodule $\SCD_X^\vee$ corresponding to the inverse Serre fuctor of $\BD(X)$
\begin{equation}\label{is}
S^{-1}(F) := F \otimes \omega_X^{-1}[-\dim X].
\end{equation} 
Let us show that this is the bimodule defined by
\begin{equation}\label{ddis}
\SCD_X^\vee(\bx,\by) = N\Hom_{\SCC_X}(\epsilon(\by),S^{-1}(\epsilon(\bx))).
\end{equation}

\begin{lemma}\label{mudd}
We have $\mu(\SCD_X^\vee) \cong \Delta_*\omega_X^{-1}[-\dim X]$.
\end{lemma}
\begin{proof}
Indeed, 
\begin{multline*}
\Res_{\kappa}(\bar\epsilon_{X\times X}(\Delta_*\omega_X^{-1}[-\dim X]))(\bx,\by) =
N\Hom_{\SCC_{X\times X}}(\kappa(\bx,\by),\Delta_*\omega_X^{-1}[-\dim X]) = \\ =
N\BC^\bullet(X\times X,(\epsilon(\by)^\vee \boxtimes \epsilon(\bx)) \otimes \Delta_*\omega_X^{-1}[-\dim X]) \cong  \\ \cong
N\BC^\bullet(X\times X,\Delta_*(\Delta^*(\epsilon(\by)^\vee \boxtimes \epsilon(\bx))\otimes \omega_X^{-1}[-\dim X])) \cong \\ \cong
N(\sd(\BC^\bullet(X,\Delta^*(\epsilon(\by)^\vee \boxtimes \epsilon(\bx))\otimes \omega_X^{-1}[-\dim X]))) \cong  
\\ \cong
N\BC^\bullet(X,\Delta^*(\epsilon(\by)^\vee \boxtimes \epsilon(\bx))\otimes \omega_X^{-1}[-\dim X]) \cong  
\\ \cong
N\BC^\bullet(X,\epsilon(\by)^\vee \otimes \epsilon(\bx)\otimes \omega_X^{-1}[-\dim X]) \cong  
N\Hom_{\SCC_X}(\epsilon(\by),S^{-1}(\epsilon(\bx))) \cong
\Res_\nabla(\SCD^\vee_X)(\bx,\by).
\end{multline*}
and the argument of Lemma~\ref{mud} completes the proof.
\end{proof}

\begin{remark}
In fact it is easy to check that
\begin{equation*}
\SCD_X^\vee \cong \RHom_{\BD(\SCD_X^\opp\otimes\SCD_X)}(\SCD_X,\SCD_X \otimes_\kk \SCD_X),
\end{equation*}
which justifies the notation.
\end{remark}

%

\subsection{Semiorthogonal decompositions}

Let $\CT$ be a triangulated category and $\CA,\CB \subset \CT$ full triangulated subcategories.
One says that
\begin{equation*}
\CT = \langle \CA,\CB \rangle
\end{equation*}
is a {\sf semiorthogonal decomposition} if 
\begin{itemize}
\item for all $A \in \CA$, $B \in B$ one has $\Hom(B,A) = 0$, and
\item for any $T \in \CT$ there is a distinguished triangle 
$T_\CB \to T \to T_\CA$
with $T_\CA \in \CA$ and $T_\CB \in \CB$.
\end{itemize}
Note that the above triangle is unique because of the semiorthogonality.

Now assume that we are given a semiorthogonal decomposition
\begin{equation*}
\BD^b(\coh(X)) = \langle \CA,\CB \rangle
\end{equation*}
with $X$ smooth and projective.
It was shown in~\cite{K-FBC} that there is a semiorthogonal decomposition
\begin{equation*}
\BD^b(\coh(X\times X)) = \langle \CA_X,\CB_X \rangle,
\end{equation*}
where the subcategories $\CA_X$ and $\CB_X$ are the minimal closed under direct summands
triangulated subcategories of $\BD^b(\coh(X\times X))$ containing all objects of the form 
$A\boxtimes F$ (resp.\ $B\boxtimes F$) with arbitrary $A \in \CA$, $B \in \CB$ and $F \in \BD^b(\coh(X))$.
Consequently, we can consider the induced decomposition of the structure sheaf of the diagonal
\begin{equation*}
Q \to \Delta_*\CO_X \to P
\end{equation*}
with $P \in \CA_X$, $Q \in \CB_X$. One can easily see that the Fourier--Mukai functors $\BD^b(\coh(X)) \to \BD^b(\coh(X))$
associated with the kernels $P$ and $Q$ take any object $F \in \BD^b(\coh(X))$ to its components $F_\CA$ and $F_\CB$
with respect to the initial semiorthogonal decomposition. Thus the above triangle can be considered as the universal 
semiorthogonal decomposition triangle.

It turns out that the kernel $Q$ has a nice interpretation in terms of the natural enhancements of the DG-categories
$\BD^b(\coh(X))$ and $\CB$. This interpretation is crucial for the rest of the paper.

%
%

Consider the \v{C}ech enhancement $\SCD = \SCD_X$ of $\BD^b(\coh(X))$.
Let $\SCA \subset \SCD$ and $\SCB \subset \SCD$ be the full DG-subcategories consisting of all objects
contained in $\CA \subset [\SCD]$ and $\CB \subset [\SCD]$ respectively.
Consider the tensor product $\SCD\lotimes_\SCB\SCD$.
Here the first factor is considered as a $\SCD$-$\SCB$-bimodule
and the second factor is considered as a $\SCB$-$\SCD$-bimodule by restricting the corresponding arguments 
of the diagonal $\SCD$-bimodule $\SCD$ to $\SCB$.
Let $\mu:\BD(\SCD^\opp\otimes\SCD) \to \BD(X\times X)$ be the equivalence discussed in the previous subsection.

\begin{proposition}\label{muq}
We have $Q \cong \mu(\SCD\lotimes_\SCB\SCD)$. 
\end{proposition}
\begin{proof}
Denote $Q' := \mu(\SCD\lotimes_\SCB\SCD) \in \BD(X\times X)$.
The composition law in $\SCD$ induces a morphism of $\SCD$-$\SCD$-bimodules 
\begin{equation}\label{dbdd}
\SCD\lotimes_\SCB\SCD \to \SCD
\end{equation} which under the functor $\mu$ gives a morphism $Q' \to \Delta_*\CO_X$ in $\BD(X\times X)$. We denote by $P'$
its cone, so that we have a distinguished triangle
\begin{equation*}
Q' \to \Delta_*\CO_X \to P'
\end{equation*}
in $\BD(X\times X)$. We need to show that $Q \cong Q'$. By definition of $Q$ and $P$ for this it is enough 
to check that $P' \in \CA_X$ and $Q' \in \CB_X$.

Note that $\SCD\lotimes_\SCB\SCD$ can be rewritten as $\SCD\lotimes_\SCB\SCB\lotimes_\SCB\SCD$.
To compute this derived tensor product we can use the bar-resolution of $\SCB$. We deduce that
\begin{equation}\label{bardbd}
\SCD\lotimes_\SCB\SCD \cong
\bigoplus_{p=0}^\infty \quad \bigoplus_{\bx_0,\dots,\bx_p \in \SCB} h_{\bx_p}\otimes\SCB(\bx_p,\bx_{p-1})\otimes\dots\otimes\SCB(\bx_1,\bx_0)\otimes h^{\bx_0}[p]
\end{equation}
where the factors $h_{\bx_p}$ and $h^{\bx_0}$ are considered as $\SCD$-modules.
Applying the functor $\mu$ we deduce that 
$Q' = \mu(\SCD\lotimes_\SCB\SCD)$ is contained in the subcategory generated by $\epsilon(\bx_0)\boxtimes\epsilon(\bx_p)^\vee$ with $\bx_0,\bx_p \in \SCB$.
In particular, $Q' \in \CB_X$.

On the other hand, for any $\bb \in \SCB$ and $\by \in \SCD$ we have
\begin{multline*}
\Hom_{\BD(\SCD^\opp\otimes\SCD)}(h_\by\otimes h^\bb,\SCD\lotimes_\SCB\SCD) \cong 
\SCD(\by,-)\lotimes_\SCB\SCD(-,\bb) \cong \\ \cong
\SCD(\by,-)\lotimes_\SCB h_\bb \cong
\SCD(\by,\bb) \cong 
\Hom_{\BD(\SCD^\opp\otimes\SCD)}(h_\by\otimes h^\bb,\SCD)
\end{multline*}
and the isomorphism is induced by the morphism~\eqref{dbdd}. Therefore the cone of that morphism
is orthogonal to all bimodules of the form $h_\by\otimes h^\bb$ in $\BD(\SCD^\opp\otimes\SCD)$. Applying the functor 
$\mu$ we conclude that the the object $P'$ is orthogonal to all objects $\epsilon(\bb)\boxtimes\epsilon(\by)^\vee$ in $\BD(X\times X)$. 
The latter generate the subcategory~$\CB_X$. Thus $P'$ is in the right orthogonal to $\CB_X$, hence $P' \in \CA_X$.
%
%
%
%
\end{proof}


\section{Normal Hocschild cohomology}

In this section $X$ is a smooth projective variety.

\subsection{Hochschild cohomology}

The {\sf Hoschschild cohomology} of a DG-category is defined as
\begin{equation*}
\HOH^\bullet(\SCD) = \Ext^\bullet_{\BD(\SCD^\opp\otimes\SCD)}(\SCD,\SCD).
\end{equation*}
For an enhanced triangulated category the Hochschild cohomology is defined as the Hochschild
cohomology of the enhancement. The Hochschild cohomology of $\BD(X)$ will be denoted by $\HOH^\bullet(X)$.

Note that using the equivalence $\mu$ one can identify 
\begin{equation}\label{hohx}
\HOH^\bullet(X) \cong \Ext^\bullet(\Delta_*\CO_X,\Delta_*\CO_X) \cong \HH^\bullet(X,\Delta^!\Delta_*\CO_X).
\end{equation}
where $\Delta^!:\BD(X\times X) \to \BD(X)$ 
\begin{equation}\label{ds}
\Delta^!(F) = \LL\Delta^*(F)\otimes\omega_X^{-1}[-\dim X]
\end{equation} 
is the right adjoint functor of $\Delta_*:\BD(X) \to \BD(X\times X)$.

Computing the RHS of~\eqref{hohx} one obtains the Hochschild--Kostant--Rosenberg isomorphism
\begin{equation*}
\HOH^t(X) = \bigoplus_{p+q=t} H^q(X,\Lambda^pT_X).
\end{equation*}
In particular, the Hochschild cohomology of $X$ lives in nonnegative degrees and $\HOH^0(X) = \kk$
if $X$ is connected.

Now assume that a semiorthogonal decomposition 
\begin{equation*}
\BD^b(\coh(X)) = \langle \CA,\CB \rangle
\end{equation*}
is given. The \v{C}ech enhancement of $\BD(X)$ induces natural enhancements of $\CA$ and $\CB$,
so we can speak about Hochschild cohomology of these categories.
Recall the induced semiorthogonal decomposition
\begin{equation*}
\BD^b(\coh(X\times X)) = \langle \CA_X, \CB_X \rangle,
\end{equation*}
and the distinguished triangle
\begin{equation}\label{dpq}
Q \to \Delta_*\CO_X \to P
\end{equation}
with $Q\in \CB_X$ and $P \in \CA_X$. Furthermore, as it was shown in~\cite{K-HH} there is an isomorphism
\begin{equation*}
\HOH^\bullet(\CA) \cong \HH^\bullet(X,\Delta^!P),
\end{equation*}
analogous to~\eqref{hohx}, and the restriction morphism $\HOH^\bullet(X) \to \HOH^\bullet(\CA)$ of Hochschild cohomology
is induced by the morphism $\Delta_*\CO_X \to P$ from~\eqref{dpq}. Consequently, one has a distinguished triangle
\begin{equation}\label{hdqtr}
\HH^\bullet(X,\Delta^!Q) \to \HOH^\bullet(X) \to \HOH^\bullet(\CA).
\end{equation}
The first term of this triangle can be thought of as a complex controlling the restriction map of Hochschild cohomology.
Our goal is to show how it can be computed in terms of the category $\CB$, especially in the case
when $\CB$ is generated by an exceptional collection.

\begin{lemma}\label{hdq}
We have $\HH^\bullet(X,\Delta^!Q) \cong \HH^\bullet(X\times X,Q \lotimes \Delta_*\omega_X^{-1}[-\dim X])$.
\end{lemma}
\begin{proof}
This evidently follows from~\eqref{ds}. Indeed, we have
\begin{multline*}
\HH^\bullet(X,\Delta^!Q) \cong 
\HH^\bullet(X,\LL\Delta^*Q\otimes \omega_X^{-1}[-\dim X]) \cong \\ \cong
\HH^\bullet(X\times X,\Delta_*(\LL\Delta^*Q\otimes \omega_X^{-1}[-\dim X])) \cong 
\HH^\bullet(X\times X,Q\lotimes \Delta_*\omega_X^{-1}[-\dim X])
\end{multline*}
(the last isomorphism is the projection formula).
\end{proof}

\subsection{The normal bimodule}

Let $\SCD = \SCD_X$ be the \v{C}ech enhancement of $\BD^b(\coh(X))$ and $\SCB \subset \SCD$ a DG-subcategory.

\begin{definition}
The {\sf normal bimodule} of the embedding $\SCB \to \SCD$ is the restriction $\SCD^\vee_\SCB$ to $\SCB$ 
of the $\SCD$-$\SCD$-bimodule $\SCD^\vee_X$ defined by~\eqref{ddis}.
%
The {\sf normal Hochschild cohomology} of $\SCB$ in $\SCD$ is defined as the derived tensor product
of the diagonal and the normal bimodule of $\SCB$
\begin{equation*}
\NHH^\bullet(\SCB,\SCD) := \SCB \lotimes_{\SCB^\opp\otimes\SCB} \SCD^\vee_\SCB.
\end{equation*}
In case when 
$\SCB$ is the induced enhancement 
of a subcategory $\CB \subset \BD^b(\coh(X))$ we will write $\NHH^\bullet(\CB,X)$ instead of $\NHH^\bullet(\SCB,\SCD)$.
\end{definition}


\begin{theorem}\label{rhhtri}
If $\BD^b(\coh(X)) = \langle \CA,\CB \rangle$ is a semiorthogonal decomposition then there is a distinguished triangle
\begin{equation}\label{rhh}
\NHH^\bullet(\CB,X) \to \HOH^\bullet(X) \to \HOH^\bullet(\CA).
\end{equation}
\end{theorem}
\begin{proof}
By Lemma~\ref{hdq} it is enough to identify $\HH^\bullet(X\times X,Q \lotimes \Delta_*\omega_X^{-1}[-\dim X])$ 
with the normal Hochschild cohomology. Recall the equivalence $\mu:\BD(\SCD^\opp\otimes\SCD) \to \BD(X\times X)$
constructed in the previous section. By Proposition~\ref{muq} we have an isomorphism 
$Q \cong \mu(\SCD\lotimes_\SCB\SCD)$ and by Lemma~\ref{mudd} an isomorphism
$\Delta_*\omega_X^{-1}[-\dim X] \cong \mu(\SCD_X^\vee)^T$. Therefore by~\eqref{mukk} we have
\begin{equation*}
\HH^\bullet(X\times X,Q \lotimes \Delta_*\omega_X^{-1}[-\dim X]) \cong
(\SCD\lotimes_\SCB\SCD)\lotimes_{\SCD^\opp\otimes\SCD} \SCD_X^\vee.
\end{equation*}
On the other hand, 
\begin{equation*}
\SCD\lotimes_\SCB\SCD \cong
\SCB \lotimes_{\SCB^\opp\otimes\SCB} (\SCD \otimes_\kk \SCD),
\end{equation*}
hence
\begin{equation*}
(\SCD\lotimes_\SCB\SCD) \lotimes_{\SCD^\opp\otimes\SCD} \SCD_X^\vee \cong
\SCB \lotimes_{\SCB^\opp\otimes\SCB} (\SCD \otimes_\kk \SCD) \lotimes_{\SCD^\opp\otimes\SCD} \SCD_X^\vee \cong
\SCB \lotimes_{\SCB^\opp\otimes\SCB} \SCD^\vee_\SCB
\end{equation*}
and this is precisely the normal Hochschild cohomology.
%
%
%
%
\end{proof}


\begin{remark}
Analogous result can be proved for arbitrary smooth strongly pretriangulated DG-category~$\SCD$. If a semiorthognal decomposition
$[\SCD] = \langle \CA,\CB \rangle$ is given and $\SCA,\SCB\subset \SCD$ are the DG-subcategories underlying $\CA$ and $\CB$ respectively, 
then there is a distinguished triangle
\begin{equation*}
\NHH^\bullet(\SCB,\SCD) \to \HOH^\bullet(\SCD) \to \HOH^\bullet(\SCA).
\end{equation*}
The proof is completely analogous to the one described here. One considers a distinguished triangle
$\SCD\lotimes_\SCB\SCD \to \SCD \to P'$ in $\BD(\SCD^\opp\otimes\SCD)$ and verifies that
$\HOH^\bullet(\SCD) = \Ext^\bullet(\SCD,\SCD) \cong \SCD\lotimes_{\SCD^\opp\otimes\SCD}\SCD^\vee$ (here the smoothness
of $\SCD$ is important) and that $\HOH^\bullet(\SCA) \cong \Ext^\bullet(P',P') \cong
\Ext^\bullet(\SCD,P') \cong P'\lotimes_{\SCD^\opp\otimes\SCD}\SCD^\vee$ (the first isomorphism here is the most subtle part;
to prove it one constructs a fully faithful embedding $\BD(\SCA^\opp\otimes\SCA) \to \BD(\SCD^\opp\otimes\SCD)$
which takes the diagonal $\SCA$ to $P'$, this embedding is not the trivial one, in fact it is the embedding induced
by the embedding $\SCA \to \SCD$ and the ``mutated'' embedding $\CA \cong {}^\perp\CB \subset [\SCD]$).
This allows to identify the cone of the morphism $\HOH^\bullet(\SCD) \to \HOH^\bullet(\SCA)$ with the (shift of)
$(\SCD\lotimes_\SCB\SCD)\lotimes_{\SCD^\opp\otimes\SCD}\SCD^\vee$ which is just the normal Hochschild cohomology.
\end{remark}

\subsection{The normal Hochschild cohomology of an exceptional collection}

Consider the case when the subcategory $\CB \subset \BD^b(\coh(X))$ is generated by an exceptional collection,
\begin{equation*}
\CB= \langle E_1,\dots,E_n \rangle.
\end{equation*}
As in the previous section we consider the \v{C}ech enhancement $\SCD$ of $\BD^b(\coh(X))$ and take
$\SCE$ to be its DG-subcategory with $n$ objects --- $E_1,\dots,E_n$. Note that $\SCE$ is a subcategory
of the DG-category $\SCB$ considered above.

By definition of the normal bimodule of $\SCE$ in $\SCD$ we have
\begin{equation}\label{sbe}
\SCD^\vee_\SCE(E_i,E_j) = \Hom_\SCD(E_j,S^{-1}(E_i)).
\end{equation} 

\begin{lemma}
The normal cohomology of $\SCE$ and $\SCB$ in $\SCD$ are the same, i.e.
$\NHH^\bullet(\SCE,\SCD) \cong \NHH^\bullet(\SCB,\SCD)$.
\end{lemma}
\begin{proof}
Follows from the fact that the restriction of DG-bimodules from $\SCB$ to $\SCE$ induces an equivalence 
$\BD(\SCB^\opp\otimes\SCB) \cong \BD(\SCE^\opp\otimes\SCE)$ compatible with the tensor product and
taking the diagonal bimodule $\SCB$ to the diagonal bimodule $\SCE$ and the normal bimodule $\SCD^\vee_\SCB$ to $\SCD^\vee_\SCE$.
\end{proof}

Our goal is to compute the normal Hochschild cohomology of $\SCE$ in $\SCD$, i.e.\
the derived tensor product of the diagonal bimodule $\SCE$ with $\SCD^\vee_\SCE$.
For this we use the bar-resolution~\eqref{bard} of $\SCE$.
Since the DG-category $\SCE$ is generated by an exceptional collection one can simplify the bar-resolution a bit. 

First note that each collection $\bx_0,\dots,\bx_p$ of objects of $\SCE$ is just a sequence $E_{a_0}, \dots, E_{a_p}$
of exceptional objects in the collection $E_1,\dots,E_n$. Thus collections $\bx_0,\dots,\bx_p \in \SCE$ are in bijection
with collections of integers $a_0,\dots,a_p \in \{1,\dots,n\}$. Since $\SCE(\bx_i,\bx_{i-1}) = \Hom_\SCD(E_{a_{i-1}},E_{a_i})$
is acyclic when $a_i < a_{i-1}$, we can omit all collections $a_0,\dots,a_p$ which have $a_{i} < a_{i-1}$ for some $i$,
thus leaving  only nondecreasing collections.
Moreover, since $\Hom_\SCD(E_a,E_a)$ is quasiisomorphic to the base field $\kk$ and so the multiplication map
\begin{equation*}
\Hom_\SCD(E_a,E_a)\otimes\Hom_\SCD(E_a,E_{a'}) \to \Hom_\SCD(E_a,E_{a'})
\end{equation*}
is a quasiisomorphism, we can omit all collections $a_0,\dots,a_p$ which are not strictly increasing.
Thus we have the following

\begin{lemma}
The diagonal bimodule $\SCE$ is quasiisomorphic to the following reduced bar-complex
\begin{equation*}
\bar\BB(\SCE) =  \bigoplus_{1\le a_0 < \dots < a_p \le n} 
\Hom_\SCD(E_{a_p},-)\otimes \Hom_\SCD(E_{a_{p-1}},E_{a_p})\otimes \dots \otimes \Hom_\SCD(E_{a_0},E_{a_1}) \otimes \Hom_\SCD(-,E_{a_0})[p],
\end{equation*}
\end{lemma}

Using the reduced bar-resolution we easily deduce

\begin{proposition}
The normal Hochschild cohomology $\NHH^\bullet(\SCE,\SCD)$ of the DG-subcategory $\SCE \subset \SCD$
generated by an exceptional collection $E_1,\dots,E_n$ is isomorphic to the cohomology 
of the bicomplex $\SCH^{\bullet,\bullet}$ with
\begin{equation*}
\SCH^{-p,q} = \bigoplus_{\substack{1\le a_0 < \dots < a_p \le n\\k_0+\dots+k_p=q}} 
\Hom^{k_0}_\SCD(E_{a_0},E_{a_1}) \otimes \dots \otimes \Hom^{k_{p-1}}_\SCD(E_{a_{p-1}},E_{a_p}) \otimes \Hom^{k_p}_\SCD(E_{a_p},S^{-1}(E_{a_0})),
\end{equation*}
and with the differential $d = d'+ d''$, where $d':\SCH^{-p,q} \to \SCH^{-p,q+1}$ is induced by the differentials of the complexes $\Hom_\SCD$ and 
$d'':\SCH^{-p,q} \to \SCH^{1-p,q}$ is induced by the multiplication maps between the adjacent factors
and the map
\begin{multline*}
\Hom^{k_p}_\SCD(E_{a_p},S^{-1}(E_{a_0})) \otimes \Hom^{k_0}_\SCD(E_{a_0},E_{a_1}) \cong \\ \cong
\Hom^{k_p}_\SCD(E_{a_p},S^{-1}(E_{a_0})) \otimes \Hom^{k_0}_\SCD(S^{-1}(E_{a_0}),S^{-1}(E_{a_1})) \to
\Hom^{k_p+k_0}_\SCD(E_{a_p},S^{-1}(E_{a_1})).
\end{multline*}
\end{proposition}

The bicomplex $\SCH$ will be referred to as {\sf the normal Hochschild bicomplex} of $\SCE$ in $\SCD$.

\begin{remark}
Alternatively, we could use the reduced bar-resolution of the normal bimodule
\begin{equation*}
\bar\BB(\SCD^\vee_\SCE) =  \bigoplus_{\substack{1\le a_0 < \dots < a_p \le n}} 
\Hom_\SCD(E_{a_p},-)\otimes \Hom_\SCD(E_{a_{p-1}},E_{a_p})\otimes \dots \otimes \Hom_\SCD(E_{a_0},E_{a_1}) \otimes 
\SCD^\vee_\SCE(E_{a_0},-) 
\end{equation*}
(note that by~\eqref{sbe} the bimodule $\SCD^\vee_\SCE$ is representable as a right $\SCE$-module, hence
it is enough to take its bar-resolution only as of a left module).
It is easy to see that tensoring it with $\SCE$ gives literally the same bicomplex $\SCH$.
\end{remark}

Consider the spectral sequence of the bicomplex $\SCH$. Its first page is obtained by taking the cohomology with respect to $d'$.
Thus
\begin{equation}\label{ehpq}
\BE_1(\SCH)^{-p,q} = 
\bigoplus_{\substack{1\le a_0 < \dots < a_p \le n\\k_0+\dots+k_p=q}} 
\Ext^{k_0}(E_{a_0},E_{a_1}) \otimes \dots \otimes \Ext^{k_{p-1}}(E_{a_{p-1}},E_{a_p}) \otimes \Ext^{k_p}(E_{a_p},S^{-1}(E_{a_0})),
\end{equation}
The differential $d_1$ is induced by the multiplication maps $m_2$. The higher differentials $d_2$, $d_3$ and so on are induced
by the higher multiplication maps $m_3$, $m_4$ and so on in the $A_\infty$ structure on $\Ext$'s induced by the DG-structure of the complexes $\Hom_\SCD$.

The spectral sequence of the normal Hochschild complex $\SCH$ will be referred to as the {\sf normal Hochschild spectral sequence}.


\section{Height of an exceptional collection}

\subsection{Height and pseudoheight}

The height and the pseudoheight of an exceptional collection $E_1,\dots,E_n$ are invariants controlling
the difference between the Hochschild cohomology of $X$ and that of the orthogonal complement 
\begin{equation}\label{ca}
\CA = \langle E_1,\dots,E_n\rangle^\perp
\end{equation}
of the collection. The height is defined in terms of the normal Hochschild cohomology.

\begin{definition}
The {\sf height} of an exceptional collection $E_1,\dots,E_n$ is defined as
\begin{equation*}
\he(E_1,\dots,E_n) = \min \{ k \in \ZZ\ |\ \NHH^k(\SCE,\SCD) \ne 0 \},
\end{equation*}
where $\SCE$ is the DG-category generated by $E_1,\dots,E_n$.
\end{definition}

We have the following simple consequence of Theorem~\ref{rhhtri}.

\begin{theorem}\label{hl}
Let $h = \he(E_1,\dots,E_n)$ be the height of an exceptional collection $E_1,\dots,E_n$ and
let $\CA$ be its orthogonal complement~\eqref{ca}. The canonical restriction morphism
$\HOH^k(X) \to \HOH^k(\CA)$ is an isomorphism for $k \le h-2$ and a monomorphism for $k = h-1$.
\end{theorem}
\begin{proof}
The long exact sequence of cohomology groups of the triangle~\eqref{rhh} gives
\begin{equation*}
\dots \to \NHH^k(\SCE,\SCD) \to \HOH^k(X) \to \HOH^k(\CA) \to \NHH^{k+1}(\SCE,\SCD) \to \dots
\end{equation*}
For $k \le h-2$ both extremal terms vanish, hence the middle arrow is an isomorphism. For $k = h-1$
the left term vanishes, hence the middle arrow is a monomorphism.
\end{proof}

\begin{corollary}
The height of an exceptional collection is invariant under mutations.
\end{corollary}
\begin{proof}
Mutations of the collection change neither the subcategory $\CA$, nor its embedding into $\BD^b(\coh(X))$.
Therefore, the morphism $\HOH^\bullet(X) \to \HOH^\bullet(\CA)$ does not change. But the height is nothing
but the maximal integer $h$ for which the statement of Theorem~\ref{hl} is true.
\end{proof}

The only drawback of the notion of height is that it may be difficult to compute. A priori its computation
requires understanding of the higher multiplications in the $\Ext$ algebra of the exceptional collection.
Below we suggest a coarser invariant, the pseudoheight, which is much easier to compute but still gives 
some control of the Hochschild cohomology.

For any two objects $F,F' \in \BD(X)$ we define their {\sf relative height} as
\begin{equation*}
\se(F,F') = \min \{ k \ |\ \Ext^k(F,F') \ne 0 \}.
\end{equation*}

\begin{definition}
The {\sf pseudoheight} of the exceptional collection $E_1,\dots,E_n$ is defined as
\begin{equation*}
\ph(E_1,\dots,E_n) = \min_{1\le a_0 < a_1< \dots < a_p \le n} \Big( \se(E_{a_0},E_{a_1})+\dots+\se(E_{a_{p-1}},E_{a_p})+\se(E_{a_p},S^{-1}(E_{a_0})) - p  \Big).
\end{equation*}
The minimum is taken over the set of all chains of indices $1 \le a_0 < a_1 < \dots < a_p \le n$.
Note that the {\sf length} $p$ of a chain enters nontrivially into the expression under the minimum. 
\end{definition}

The pseudoheight gives a lower bound for the height.

\begin{lemma}\label{hph}
We have $\he(E_1,\dots,E_n) \ge \ph(E_1,\dots,E_n)$.
\end{lemma}
\begin{proof}
The pseudoheight is the minimal total degree of nontrivial terms of the first page of the normal Hochschild spectral sequence~\eqref{ehpq}.
Since the spectral sequence converges to normal Hochschild cohomology $\NHH^\bullet(\SCE,\SCD)$, we conclude that the latter is zero in degrees
$k < \ph(E_1,\dots,E_n)$. Therefore $\he(E_1,\dots,E_n) \ge \ph(E_1,\dots,E_n)$.
\end{proof}

Since the pseudoheight is not greater than the height, it gives the same restriction on the morphism of Hochschild cohomology.

\begin{corollary}\label{phl}
Let $h = \ph(E_1,\dots,E_n)$ be the pseudoheight of an exceptional collection $E_1,\dots,E_n$ and
let $\CA$ be its orthogonal complement~\eqref{ca}. The canonical restriction morphism
$\HOH^k(X) \to \HOH^k(\CA)$ is an isomorphism for $k \le h-2$ and a monomorphism for $k = h-1$.
\end{corollary}

Somtimes one can easily show that the pseudoheight equals the height.

\begin{proposition}\label{heph}
Assume that the pseudoheight $\ph(E_1,\dots,E_n)$ is achieved on a chain of length $0$, i.e.\
$\ph(E_1,\dots,E_n) = \se(E_i,S^{-1}(E_i))$ for some $i$. Then
$\he(E_1,\dots,E_n) = \ph(E_1,\dots,E_n)$.
\end{proposition}
\begin{proof}
Let $k = \ph(E_1,\dots,E_n)$ and assume that $\se(E_i,S^{-1}(E_i)) = k$.
Note that we have an inclusion $\Ext^k(E_i,S^{-1}(E_i)) \subset \BE_1(\SCH)^{0,k}$, hence all the higher differentials
of the normal Hochschild spectral sequence vanish on this space (just because the higher differentials increase $p$ and $\SCH$ 
is concentrated in nonpositive degrees with respect to $p$). On the other hand, all differentials increase the total degree $p+q$
and by assumption $k$ is the minimal total degree of nonzero elements of the first page of the spectral sequence.
Hence no nontrivial differentials of the spectral sequence have target at $\Ext^k(E_i,S^{-1}(E_i))$,
hence it embedds into $\NHH^k(\SCE,\SCD)$. Thus $\he(E_1,\dots,E_n) \le k = \ph(E_1,\dots,E_n)$.
On the other hand we know that $\he(E_1,\dots,E_n) \ge \ph(E_1,\dots,E_n)$ by Lemma~\ref{hph}. Thus $\he(E_1,\dots,E_n) = \ph(E_1,\dots,E_n)$.
\end{proof}

\subsection{Formal deformation spaces}

We refer to~\cite{KS} for generalities about deformation theory of $A_\infty$-algebras and categories.
Recall that the second Hochschild cohomology is the tangent space to the deformation space of a category
and the third Hochschild cohomology is the space of obstructions. By Theorem~\ref{hl} 
if $\he(E_1,\dots,E_n) \ge 4$ then the map $\HOH^2(X) \to \HOH^2(\CA)$ is an isomorphism 
and the map $\HOH^3(X) \to \HOH^3(\CA)$ is a monomorphism. As a consequence we have

\begin{proposition}
If $\he(E_1,\dots,E_n) \ge 4$ then the formal deformation spaces of categories $\BD(X)$ and $\CA$ are isomorphic.
\end{proposition}
\begin{proof}
The formal deformation space of a DG-category $\SCD$ is described in terms of the Gerstenhaber algebra structure 
on the Hochschild cohomology of $\SCD$. To be more precise only the cohomology in degrees up to $3$ play role.
Now let $\SCD$ be the \v{C}ech enhancement of $\BD(X)$ and $\SCA$ the induced enhancement of~$\CA$. 
The restriction morphism $\HOH^\bullet(\SCD) \to \HOH^\bullet(\SCA)$ is a morphism of Gerstenhaber
algebras (this is clear from the explicit formulas for the multiplication and the bracket)
which is an isomorphism in degrees up to $2$ and a monomorphism in degree $3$. Therefore
the formal deformation spaces are isomorphic.
\end{proof}


\subsection{Anticanonical pseudoheight}

Sometimes it is more convenient to replace the inverse Serre functor in the definition of the height
and pseudoheight by the anticanonical twist. Of course, this just shifts the result by $\dim X$.

\begin{definition}
The {\sf anticanonical pseudoheight} of an exceptional collection $E_1,\dots,E_n$ is defined as 
\begin{equation*}
\ph_\ac(E_1,\dots,E_n) = \min_{1\le a_0 < a_1< \dots < a_p \le n} \Big( \se(E_{a_0},E_{a_1})+\dots+\se(E_{a_{p-1}},E_{a_p})+\se(E_{a_p},E_{a_0}\otimes\omega_X^{-1}) - p  \Big).
\end{equation*}
Thus $\ph_\ac = \ph - \dim X$.
The {\sf anticanonical height} 
is defined as
$\he_\ac(E_1,\dots,E_n) = \he(E_1,\dots,E_n) - \dim X$.
\end{definition}

Let $E_1,\dots,E_n$ be an exceptional collection. The collection
\begin{equation*}
E_1,\dots,E_n,E_1\otimes\omega_X^{-1},\dots,E_n\otimes\omega_X^{-1}
\end{equation*}
will be called the {\sf (anticanonically) extended collection}.
We will say that the extended collection is {\sf $\Hom$-free} if $\Ext^p(E_i,E_j) = 0$ for $p \le 0$ and all $1\le i < j \le i+n$.
A $\Hom$-free collection is called {\sf cyclically $\Ext^1$-connected} if there is a chain $1 \le a_0 < a_1 < \dots < a_{p-1} < a_p \le n$ such that 
$\Ext^1(E_{a_s},E_{a_{s+1}}) \ne 0$ for all $s = 0,1,\dots,p-1$ and $\Ext^1(E_{a_p},E_{a_0}\otimes\omega_X^{-1}) \ne 0$.

\begin{lemma}\label{h12}
If $E_1,\dots,E_n$ is an exceptional collection such that the anticanonically extended collection is $\Hom$-free 
then $\ph_\ac(E_1,\dots,E_n) \ge 1$ and $\ph(E_1,\dots,E_n) \ge 1 + \dim X$.
If, in addition, the extended collection is not cyclically $\Ext^1$-connected then $\ph_\ac(E_1,\dots,E_n) \ge 2$
and $\ph(E_1,\dots,E_n) \ge 2 + \dim X$.
\end{lemma}
\begin{proof}
If the extended collection is $\Hom$-free then $\se(E_{a_i},E_{a_{i+1}}) \ge 1$ and $\se(E_{a_p},E_{a_0}\otimes\omega_X^{-1}) \ge 1$, hence 
\begin{equation*}
\se(E_{a_0},E_{a_1}) + \dots + \se(E_{a_{p-1}},E_{a_p}) + \se(E_{a_p},E_{a_0}\otimes\omega_X^{-1}) - p \ge (p+1) - p = 1,
\end{equation*}
hence the anticanonical pseudoheight is not less than $1$. If, moreover, the extended collection is not cyclically $\Ext^1$-connected then in the sum
$\se(E_{a_0},E_{a_1}) + \dots + \se(E_{a_{p-1}},E_{a_p}) + \se(E_{a_p},E_{a_0}\otimes\omega_X^{-1})$ at least one summand is at least $2$, 
hence the sum is at least $p + 2$, hence the LHS above is at least $2$, hence the anticanonical pseudoheight is not less than $2$.
\end{proof}

\section{Examples}\label{s-ex}

In this section we provide several examples which show that the anticanonical height is easily computable.
All these examples deal with quasiphantom categories constructed recently in the derived categories of 
some surfaces of general type.
In all examples below we will use the following simple observation.

\begin{lemma}\label{h0}
Let $X$ be a surface with ample canonical class $K_X$ and $(L_1,L_2)$ a pair of line bundles. If 
\begin{equation*}
c_1(L_1)\cdot K_X \ge c_1(L_2)\cdot K_X, 
\end{equation*}
then $\Hom(L_1,L_2) = 0$ unless $L_1 \cong L_2$.
\end{lemma}
\begin{proof}
A nonzero morphism $L_1 \to L_2$ gives a section of the line bundle $L_1^{-1}\otimes L_2$, hence
$L_1^{-1}\otimes L_2 \cong \CO_X(D)$ for an effective divisor $D$. Since $K_X$ is ample we have $D\cdot K_X > 0$
unless $D = 0$. This proves the Lemma.
\end{proof}


Another useful observation is the following

\begin{lemma}\label{h2ac}
Let $X$ be a smooth projective surface with $H^2(X,\omega_X^{-1}) \ne 0$. If $E_1,\dots,E_n$ is an exceptional
collection consisting of line bundles then $\ph_\ac(E_1,\dots,E_n) \le 2$. Moreover, if
$\ph_\ac(E_1,\dots,E_n) = 2$ then $\he_\ac(E_1,\dots,E_n) = 2$ as well.
\end{lemma}
\begin{proof}
Since $E_i$ is a line bundle we have $\Ext^k(E_i,E_i\otimes\omega_X^{-1}) = H^k(X,E_i^\vee\otimes E_i\otimes\omega_X^{-1}) = H^k(X,\omega_X^{-1})$ hence 
$\ph_\ac(E_1,\dots,E_n) \le \se(E_i,E_i\otimes\omega_X^{-1}) \le 2$. If the anticanonical pseudoheight is $2$ then by Proposition~\ref{heph}
the anticanonical height is also $2$.
\end{proof}

\subsection{Burniat surfaces}

Alexeev and Orlov in~\cite{AO} have constructed an exceptional collection of length~$6$ in the derived category of a Burniat surface 
(Burniat surfaces is a family of surfaces of general type with $p_g = q = 0$ and $K^2 = 6$). The orthogonal subcategory
$\CA\subset \BD(X)$ has trivial Hochschild homology and $K_0(\CA) = (\ZZ/2\ZZ)^6$.

\begin{proposition}
The height of the Alexeev--Orlov exceptional collection is $4$.
\end{proposition}
\begin{proof}
The collection consists of line bundles, the canonical degrees of the anticanonically extended collection are given by the following sequence
\begin{equation*}
3,3,2,2,2,0\ ;\ -3,-3,-4,-4,-4,-6,
\end{equation*}
the semicolon separates the extended part.
In particular, the degrees do not increase, so by Lemma~\ref{h0} the collection is $\Hom$-free. 
Therefore $\ph_\ac(E_1,\dots,E_6) \ge 1$ by Lemma~\ref{h12}. Let us also check
%
that the collection is not cyclically $\Ext^1$-connected. In other words for any chain $1 \le a_0 < \dots < a_p \le 6$ 
we have to check that either at least for one $i$ we have $\Ext^1(E_{a_i},E_{a_{i+1}}) = 0$ or $\Ext^1(E_{a_p},E_{a_0}\otimes\omega_X^{-1}) = 0$.
If $p > 0$ then $\Ext^1(E_{a_0},E_{a_1}) = 0$ by Lemma 4.8 of \cite{AO}.
If $p = 0$ then since $E_{a_0}$ is a line bundle we have
\begin{equation*}
\Ext^1(E_{a_0},E_{a_0}\otimes\omega_X^{-1}) = H^1(X,\omega_X^{-1})
\end{equation*}
and it is easy to see that it is also $0$. Thus $\ph_\ac(E_1,\dots,E_6) \ge 2$ by Lemma~\ref{h12}. 
By Lemma~\ref{h2ac} we conclude that $\ph_\ac(E_1,\dots,E_6) = \he_\ac(E_1,\dots,E_6) = 2$.
\end{proof}

\begin{corollary}
If $X$ is a Burniat surface then the natural restriction morphism $\HOH^k(X) \to \HOH^k(\CA)$ 
is an isomorphism for $k \le 2$ and a monomorphism for $k = 3$. In particular, the formal deformation
space of $\CA$ is isomorphic to that of $\BD(X)$.
\end{corollary}

\subsection{The Beauville surface}

Galkin and Shinder in~\cite{GS} have constructed six different exceptional collections of length $4$ in the derived category of the Beauville surface 
(Beauville surface is an example of a surface of general type with $p_g = q = 0$ and $K^2 = 8$). The orthogonal subcategory $\CA\subset \BD(X)$ 
to any of those has trivial Hochschild homology and $K_0(\CA) = (\ZZ/5\ZZ)^2$. We will consider the first
two exceptional collections (called $I_1$ and $I_0$ in Theorem~3.5 of loc.\ cit.), since the other four 
can be obtained from them by taking the anticanonically extended collection and then restricting to its
length 4 subcollections (which are automatically exceptional due to Serre duality). This operation clearly 
does not change the pseudoheight.

\begin{proposition}
The pseudoheight of the collection $I_1$ is $4$ and of the collection $I_0$ is $3$.
The height of both collections is $4$.
\end{proposition}
\begin{proof}
The collections consist of line bundles, the canonical degrees of the anticanonically extended collections are given by the following sequences
\begin{equation*}
\begin{array}{rl}
I_1:\quad & 0,-2,-2,-4\ ;\ -8,-10,-10,-12\\
I_0:\quad & 0,-2,-4,-6\ ;\ -8,-10,-12,-14
\end{array}
\end{equation*}
The degrees do not increase, so by Lemma~\ref{h0} the collections are $\Hom$-free. Thus $\ph_\ac(E_1,\dots,E_4) \ge 1$ by Lemma~\ref{h12}.
Let us check whether the collections are $\Ext^1$-connected. It turns out that the collection $I_0$ is $\Ext^1$-connected by the chain
$a=(1,2,3,4)$, hence $\ph_\ac(I_0) = 1$, while the collection $I_1$ is not $\Ext^1$-connected, 
since there is no $\Ext^1$ from $E_1,\dots,E_4$  to $E_1\otimes\omega_X^{-1},\dots,E_4\otimes\omega_X^{-1}$ ---
this can be easily seen from the character matrices in Proposition~3.7 of loc.\ cit., hence $\ph_\ac(I_1) = 2$.
Again using Lemma~\ref{h2ac} we conclude that $\he_\ac(I_1)=2$.

To show that the anticanonical height of the collection $I_0$ is also $2$, or equivalently that its height is~$4$,
we need to look at the normal Hochschild spectral sequence. It is clear that the only component of total degree $3$ in $\BE_1(\SCH)$ is
\begin{equation*}
\Ext^1(E_1,E_2)\otimes \Ext^1(E_2,E_3)\otimes \Ext^1(E_3,E_4)\otimes \Ext^3(E_4,S^{-1}(E_1)) \subset \BE_1(\SCH)^{-3,6}.
\end{equation*}
The differential $d_1$ from this term lands into the direct sum of the following four terms
\begin{equation*}
\begin{array}{c}
\Ext^2(E_1,E_3)\otimes \Ext^1(E_3,E_4)\otimes \Ext^3(E_4,S^{-1}(E_1)),\\ 
\Ext^1(E_1,E_2)\otimes \Ext^2(E_2,E_4)\otimes \Ext^3(E_4,S^{-1}(E_1)),\\ 
\Ext^1(E_1,E_2)\otimes \Ext^1(E_2,E_3)\otimes \Ext^4(E_3,S^{-1}(E_1)),\\ 
\Ext^1(E_2,E_3)\otimes \Ext^1(E_3,E_4)\otimes \Ext^4(E_4,S^{-1}(E_2)), 
\end{array}
\end{equation*}
which is a subspace of $\BE_1(\SCH)^{-2,6}$. The map into each term is just the multiplication.
It is easy to check that even the first map $\Ext^1(E_1,E_2)\otimes\Ext^1(E_2,E_3) \to \Ext^2(E_1,E_3)$
is nonzero, hence $\BE_2(\SCH)^{-3,6} = 0$ and it follows that $\he(E_1,\dots,E_4) \ge 4$.
Finally, using the argument of Proposition~\ref{heph} it is easy to see that $\BE_\infty(\SCH)^{0,4} \ne 0$,
so $\he(E_1,\dots,E_4) = 4$.
%
%
\end{proof}

\begin{corollary}
Let $\CA_0$ and $\CA_1$ be the orthogonal subcategories of the exceptional collections $I_0$ and $I_1$ 
on the Beauville surface. Then the natural restriction morphisms $\HOH^k(X) \to \HOH^k(\CA_{t})$
are isomorphisms for $k \le 2$ and monomorphisms for $k = 3$.
%
\end{corollary}

In fact the Hochschild cohomology of the Beauville surface is very simple (see Lemma 2.6 of \cite{GS})
\begin{equation*}
\HOH^0(X) = \kk,\quad
\HOH^1(X) = \HOH^2(X) = 0,\quad
\HOH^3(X) = \kk^6,\quad
\HOH^4(X) = \kk^9.
\end{equation*}
It follows that $\HOH^0(\CA_{0,1}) =  \kk$ and $\HOH^1(\CA_{0,1}) = 0$,
$\HOH^2(\CA_{0,1}) = 0$.
It is an interesting question whether the quasiphantom categoies $\CA_0$ and $\CA_{1}$ are equivalent.

\subsection{The classical Godeaux surface}

B\"ohning--Graf von Bothmer--Sosna in~\cite{BBS} have constructed an exceptional collection 
of line bundles of length $n = 11$ in the derived category of the classical Godeaux surface. 
The orthogonal subcategory $\CA \subset \BD(X)$ has trivial Hochschild homology and $K_0(\CA) = \ZZ/5\ZZ$. 
Historically, this is the first example of a quasiphantom category. However, it is more complicated
than the other two examples, which is the reason why we considered those first. To find the height
of the B\"ohning--Graf von Bothmer--Sosna collection we will need the following 

\begin{lemma}\label{e23}
If $E_1,\dots,E_{11}$ is the B\"ohning--Graf von Bothmer--Sosna collection then 
\begin{equation*}
\Hom(E_3,E_2(-K)) = \Ext^1(E_3,E_2(-K)) = 0
\quad\text{and}\quad
\Ext^2(E_3,E_2(-K)) \ne 0.
\end{equation*}
\end{lemma}
\begin{proof}
By Serre duality we should compute $\Ext^p(E_2,E_3(2K)) = H^p(X,E_2^{-1}\otimes E_3(2K))$.
Recall that $\Hom(E_2,E_3)\ne 0$, hence $E_2^{-1}\otimes E_3 \cong \CO_X(D)$ for some effective divisor $D$.
Moreover, by Lemma 8.2 and 10.2 of~\cite{BBS} we have $D\cdot K_X = 1$ and $\chi(\CO_X(D)) = -1$,
hence from the sequence
\begin{equation*}
0 \to \CO_X \to \CO_X(D) \to \CO_D(D) \to 0
\end{equation*}
we conclude that $\chi(\CO_D(D)) = -2$, so it follows 
from the results of section 4 of~\cite{BBS} that $D$ is one of 15 lines on $X$. In paticular, 
$D \cong \PP^1$ and $\CO_D(D) = \CO_D(-3)$.
Twisting the above sequence by $2K$ we obtain
\begin{equation*}
0 \to \CO_X(2K) \to \CO_X(D + 2K) \to \CO_D(-1) \to 0.
\end{equation*}
Since $\CO_X(2K)$ has only $H^0$ which is of dimension $2$ and $\CO_D(-1)$ has no cohomology at all,
we conclude that $\CO_X(D+2K)$ has only $H^0$ which is also of dimension 2.
\end{proof}

\begin{proposition}
The pseudoheight of the B\"ohning--Graf von Bothmer--Sosna exceptional collection is~$3$
and its height is $4$.
\end{proposition}
\begin{proof}
By Lemma 10.2 of~\cite{BBS} and Lemma~\ref{e23} above we have
\begin{equation*}
\se(E_2,E_3) + \se(E_3,E_2(-K)) - 1 = 0 + 2 - 1 = 1,
\end{equation*}
hence $\ph_\ac(E_1,\dots,E_{11}) \le 1$.
So, we only have to check that $\ph_\ac(E_1,\dots,E_{11}) > 0$.

Note that the extended collection is almost $\Hom$-free. 
Indeed, all objects of the anticanonically extended collection are line bundles of the following canonical degrees 
\begin{equation*}
0,0,1,0,0,0,1,0,0,0,0\ ;\ -1,-1,0,-1,-1,-1,0,-1,-1,-1,-1.
\end{equation*}
By Lemma~\ref{h0} we could only have morphisms between the objects of a nonextended collection
(or their twists), which by~\cite[Lemma 10.2]{BBS} are nonzero only from $E_2$ to $E_3$.
%
%

Therefore, to show that the anticanonical pseudoheight is positive
we only have to check that there are no chains $1 \le a_0 < \dots < a_p \le 11$
such that in the chain of line bundles $E_{a_0},\dots,E_{a_p},E_{a_0}(-K)$ there is an
adjacent pair which has $\Hom$, and all other adjacent pairs have $\Ext^1$ between them.

Assume such a chain exists.
The arguments above show that the pair $(E_2,E_3)$ should be in the chain.
In particular we should have $a_0 \le 2$.
Assume first that $a_0 = 1$. Then we must have $a_1 = 2$, $a_2 = 3$. 
But $\Ext^1(E_1,E_2) = 0$, so this a contradiction. 
Assume that $a_0 = 2$. Then $a_1 = 3$. Note that from $E_3$ to $E_4,\dots, E_{11}$ there are no $\Ext^1$, 
hence $p$ should be $1$. But $\Ext^1(E_3,E_2(-K)) = 0$ by Lemma~\ref{e23}, so this is again a contradiction.
Thus the anticanonical pseudoheight is positive and we are done.

To compute the height we have to look at the spectral sequence.
By the arguments above we already know that $\he(E_1,\dots,E_{11}) \ge \ph(E_1,\dots,E_{11}) = \ph_\ac(E_1,\dots,E_{11}) + 2 = 3$. On the other hand,
by the arguments of Proposition~\ref{heph} it is easy to show that $\he(E_1,\dots,E_{11}) \le 4$.
So, the only thing to check is whether the Hochschild homology of the normal bimodule is nontrivial
in degree $3$. For this one has to analyze all chains $a_0 < \dots < a_p$ on which the pseudoheight
is achieved and to analyze the products (and maybe the higher products) of their terms. 
We already know that the pseudoheight is achieved on the chain $(2,3)$. But the argument of Lemma~\ref{h0}
shows that the composition 
\begin{equation*}
\Hom(E_2,E_3)\otimes\Ext^2(E_3,E_2(-K)) \to \Ext^2(E_2,E_2(-K))
\end{equation*}
is a monomorphism, hence this term is killed in the second page of the spectral sequence. On the other hand, 
the only other chains on which the pseudoheight might be achieved are
\begin{equation*}
(1,3), (1,7), (2,7), (4,7), (5,7), (6,7).
\end{equation*}
An explicit computation~\cite{BB} shows that this is not the case for all of them, hence the height is $4$.
\end{proof}

\begin{corollary}
The natural restriction morphism $\HOH^k(X) \to \HOH^k(\CA)$ is an 
isomorphism for $k \le 2$ and a monomorphism for $k = 3$ for the classical Godeaux surface.
In particular, the formal deformation
space of $\CA$ is isomorphic to that of $\BD(X)$.
\end{corollary}

%
%

\section{The necessary and sufficient conditions of fullness}

Quite unexpectedly, the height gives a necessary condition of fullness of an exceptional collection.

\begin{proposition}
Let $X$ be a smooth projective variety and let $E_1,\dots,E_n$ be an exceptional collection in $\BD^b(\coh(X))$.
If $\he(E_1,\dots,E_n) > 0$ then the collection is not full.
\end{proposition}
\begin{proof}
Let $\CA$ be the orthogonal subcategory. If $\he(E_1,\dots,E_n) > 0$ then by Theorem~\ref{hl} 
the morphism $\HOH^0(X) \to \HOH^0(\CA)$ is a monomorphism. Since $\HOH^0(X) = H^0(X,\CO_X) \ne 0$,
we conclude that $\HOH^0(\CA) \ne 0$. Therefore $\CA \ne 0$ so the collection is not full.
\end{proof}

Since the height is not smaller than the pseudoheight we obtain also an easily verifiable criterion of nonfullness.

\begin{corollary}
If $\ph(E_1,\dots,E_n) > 0$ then the collection is not full.
\end{corollary}

In particular, it follows that in all examples of section~\ref{s-ex} the collections are not full.
Note that this argument does not use the computation of the Grothendieck group of these categories
and can be applied also for real phantom categories, when the Grothendieck
group does not help. 

Even more unexpectedly is that one can also use the normal Hochschild cohomology to deduce
the fullness of the collection. Let
\begin{equation*}
\rho:\NHH^\bullet(\SCE,\SCD) \to \HOH^\bullet(X)
\end{equation*}
be the morphism of~\eqref{rhh}.

\begin{theorem}
Let $X$ be a smooth and connected projective variety and let $E_1,\dots,E_n$ be an exceptional collection in $\BD^b(\coh(X))$.
Assume that there is an element $\xi \in \NHH^0(\SCE,\SCD)$ such that $\rho(\xi) \ne 0$.
Then the collection $E_1,\dots,E_n$ is full.
\end{theorem}
\begin{proof}
Again, let $\CA$ be the orthogonal subcategory. 
Note that $\HOH^0(X) = H^0(X,\CO_X) = \kk$ since $X$ is connected. Therefore, after rescaling we can assume
that $\rho(\xi) = 1_X \in \HOH^0(X)$. Since~\eqref{rhh} is exact, the image of $1_X$ under the restriction 
morphism $\HOH^0(X) \to \HOH^0(\CA)$ is zero. On the other hand, it is clear that the image equals 
$1_\CA \in \HOH^0(\CA)$. We conclude that $1_\CA = 0$ in $\HOH^\bullet(\CA)$. But this means that $\CA = 0$, 
hence the collection is full.
\end{proof}

Of course, to use this criterion one needs a method to check that $\rho(\xi) \ne 0$. 
Note that for each object $E \in \BD^b(\coh(X))$
there is a canonical evaluation morphism
\begin{equation*}
\ev_E:\HOH^0(X) \to \Hom(E,E).
\end{equation*}
For connected $X$ we know that $\HOH^0(X) = \kk$ and $\ev_E(\lambda) = \lambda \cdot \id_E \in \Hom(E,E)$. 
So to check that $\rho(\xi) \ne 0$ it is enough to find an object $E$ such that $\ev_E(\rho(\xi)) \ne 0$.
It is quite natural to take for $E$ one of the objects $E_i$.

Let $\eta_i \in \Hom(S^{-1}(E_i),E_i)$ be the generator of
\begin{equation*}
\Hom(S^{-1}(E_i),E_i) \cong \Hom(E_i,E_i)^\vee = \kk^\vee = \kk.
\end{equation*}

Recall that the normal Hochschild cohomology $\NHH^\bullet(\SCE,\SCD)$ is computed 
by the spectral sequence $\BE_r(\SCH)^{-p,q}$. In particular, it follows that $\NHH^0(\SCE,\SCD)$ has a filtration with factors
$\BE_\infty(\SCH)^{-p,p}$. Thus $\BE_\infty(\SCH)^{0,0}$ is a subspace in $\NHH^0(\SCE,\SCD)$
and $\BE_\infty(\SCH)^{1-n,n-1}$ is the quotient of $\NHH^0(\SCE,\SCD)$.

\begin{proposition}\label{erx}
Let $p$ be the maximal integer such that $\BE_\infty(\SCH)^{-p,p} \ne 0$. Let 
\begin{equation*}
\xi_p = \sum \xi_p^{a,k} \in \bigoplus_{a,k}
\Ext^{k_0}(E_{a_0},E_{a_1}) \otimes \dots \otimes \Ext^{k_{p-1}}(E_{a_{p-1}},E_{a_p}) \otimes \Ext^{k_p}(E_{a_p},S^{-1}(E_{a_0})) \subset \BE_1(\SCH)^{-p,p}
\end{equation*}
(the sum is over the set of all chains $a = (1\le a_0 < \dots < a_p \le n)$ of length $p+1$ and over all
collections of integers $k = (k_0,\dots,k_p)$ such that $k_0 + \dots + k_p = p$) 
be a lift of a nontrivial element $\xi \in \BE_\infty(\SCH)^{-p,p}$.
Let $\xi_p^i$ be the sum of those $\xi_p^{a,k}$ for which $a_0 = i$.
Then up to a nonzero constant
\begin{equation*}
\ev_{E_i}(\rho(\xi)) = m_{p+2}(\xi_p^{i}\otimes \eta_i).
\end{equation*}
In particular, if $m_{p+2}(\xi_p^i\otimes\eta_{i}) \ne 0$ for some $i$, then the collection $E_1,\dots,E_n$ is full.
%
\end{proposition}
\begin{proof}
Evaluation morphism is the composition
\begin{multline*}
\HOH^\bullet(X) \cong 
\Ext^\bullet(\Delta_*\CO_X,\Delta_*\CO_X) \to 
\Ext^\bullet(E\boxtimes E^\vee,\Delta_*\CO_X) \cong \\ \cong
\Ext^\bullet(\LL\Delta^*(E\boxtimes E^\vee),\CO_X) \cong
\Ext^\bullet(E \lotimes E^\vee,\CO_X) \cong
\Ext^\bullet(E,E).
\end{multline*}
where the first morphism is induced by the natural map $\tr_E:E\boxtimes E^\vee \to \Delta_*\CO_X$
(which corresponds under the above chain of morphisms to $\id_E \in \Hom(E,E)$). On the other hand,
since $\Delta_*\CO_X$ is a perfect complex on $X\times X$ we have an isomorphism
\begin{equation*}
\Ext^\bullet(\Delta_*\CO_X,\Delta_*\CO_X) \cong
\HH^\bullet(X\times X,\Delta_*\CO_X \lotimes (\Delta_*\CO_X)^\vee),
\end{equation*}
and by Grothendieck duality we have
\begin{equation*}
(\Delta_*\CO_X)^\vee \cong \Delta_*\Delta^!\CO_{X\times X} \cong \Delta_*\omega_X^{-1}[-\dim X].
\end{equation*}
Dualizing the morphism $\tr_E$ we obtain a morphism
\begin{equation*}
\eta_E:\Delta_*\omega_X^{-1}[-\dim X] \to E^\vee\boxtimes E
\end{equation*}
which combined with a sequence of isomorphisms
\begin{equation*}
\HH^\bullet(X\times X,\Delta_*\CO_X \lotimes (E^\vee \boxtimes E)) \cong
\HH^\bullet(X,\LL\Delta^*(E^\vee \boxtimes E)) \cong
\HH^\bullet(X,E^\vee \lotimes E) \cong
\Ext^\bullet(E,E)
\end{equation*}
gives a commutative diagram
\begin{equation*}
\xymatrix{
\Ext^\bullet(\Delta_*\CO_X,\Delta_*\CO_X) \ar[d]_{\ev_E} \ar[r]^-\cong &
\HH^\bullet(X\times X,\Delta_*\CO_X \lotimes (\Delta_*\CO_X)^\vee) \ar[d]^{\eta_E} \\
\Ext^\bullet(E,E) \ar@{=}[r] &
\Ext^\bullet(E,E) 
}
\end{equation*}
Translating this into the derived category $\BD(\SCD^\opp\otimes\SCD)$ of $\SCD$-$\SCD$-bimodules via
the equivalence $\mu$ we obtain a commutative diagram
\begin{equation*}
\xymatrix{
\HOH^\bullet(\SCD)
\ar[r]^-{\cong} \ar[d]_{\ev_\bx} &
\SCD \lotimes_{\SCD^\opp\otimes\SCD} \SCD^\vee \ar[d]^{1 \otimes \eta_\bx} \\
\Hom_\SCD(\bx,\bx) \ar@{=}[r] &
\Hom_\SCD(\bx,\bx) 
}
\end{equation*}
where $\eta_\bx:\SCD^\vee \to h_\bx\otimes h^\bx$ is the canonical morphism. 

Thus we have to identify the morphism $\eta_\bx$ as a morphism from the (reduced) bar-resolution of $\SCD^\vee$,
restrict it to the bar-resolution of $\SCD^\vee_\SCE$ and then tensor it with $\SCE$.
This turns out to be a difficult question, and we are not able to write down a general answer.
However, we will write down a simpler morphism which is enough for our purposes. 

Assume that $\epsilon(\bx) = E$ is an exceptional object.
Then in $\BD(\SCD^\opp\otimes\SCD) \cong \BD(X\times X)$ there is a unique morphism from $\SCD^\vee$
to $h_\bx\otimes h^\bx$. This morphism restricts to a nontrivial morphism
\begin{equation*}
\SCD^\vee(\bx,-) \xrightarrow{\ \eta_\bx\ } (h_\bx\otimes h^\bx)(\bx,-) = \Hom_\SCD(E,E)\otimes h^\bx \cong h^\bx.
\end{equation*}
By definition of the bimodule $\SCD^\vee$ the image under $\epsilon$ of the above composition is a morphism $S^{-1}(E) \to E$ 
which is unique as $E$ is exceptional. On the other hand, let $\bar\eta_E$ be a closed element of degree $0$ in 
$\Hom_\SCD(S^{-1}(E),E)$. By definition of $\SCD^\vee$ the multiplication by $\bar\eta_E$ defines a nontrivial map 
$\SCD^\vee(\bx,-) \to h^\bx$ which thus has to be equal to the map $\eta_E$. Thus we have showed that the composition
$\ev_E\circ\rho$ coincides up to a nonzero constant with the map induced by the $\bar\eta_E$ multiplication. 
It follows that in the spectral sequence the induced map is the map $m_{p+2}(-\otimes\eta_E)$.
%
%
%
%
%
%
%
%
%
%
%
\end{proof}

As an example consider $X = \PP^{n-1}$, the projective space, and $(E_1,\dots,E_n) = (\CO,\dots,\CO(n-1))$.
Then $\Ext^k(E_i,E_j) \ne 0$ only for $k = 0$ and $\Ext^k(E_j,S^{-1}(E_i)) = \Ext^{k+1-n}(E_j,E_i(n)) \ne 0$ only for $k = n-1$.
Thus $\BE_1(\SCH)^{-p,q} \ne 0$ only for $q = n-1$. Therefore the only nontrivial differential in the spectral sequence $d_1$
is given by the multiplication $m_2$, the spectral sequence degenerates in the second term, and so $\NHH^\bullet(\SCE,\SCD)$
is the cohomology of the complex
\begin{equation*}
V^{\otimes n} \to \bigoplus_{i=1}^{n} V^{\otimes i-1} \otimes S^2V \otimes V^{\otimes n-i-1} \to \dots
\end{equation*}
(the last summand in the second term corresponds to symmetrization of the last and the first factors of the first term).
Therefore, $\NHH^0(\SCE,\SCD)$ is the set of all $\xi \in V^{\otimes n}$ such that the symmetrization of $\xi$
with respect to any pair of cyclically adjacent indices is zero. Thus $\xi$ should be completely antisymmetric, that is
$\xi \in \Lambda^n V \subset V^{\otimes n}$. One can check that $m_{n+1}(\xi\otimes\eta_1) \ne 0$ which gives yet another 
proof of the fact that the Beilinson collection is full.

The element $\xi$ looks to be closely related to the {\em quantum determinant} considered in~\cite{BP}.
It would be very interesting to understand the relationship.

An advantage of this criterion of fullness is the fact that to check fullness of a collection one just
has to guess appropriate $\xi$. After that only two things should be checked. First, that $\xi$ is a cocylcle,
which means that multiplications $m_2,\dots,m_{p+1}$ vanish on $\xi$, and second, that $m_{p+2}(\xi\otimes\eta_1)$
does not vanish.

\end{document}